\RequirePackage{booktabs}
\documentclass[pdflatex,sn-mathphys,Numbered]{sn-jnl}

\newcommand{\bea}{\begin{eqnarray}}
\newcommand{\eea}{\end{eqnarray}}
\newcommand{\bean}{\begin{eqnarray*}}
\newcommand{\eean}{\end{eqnarray*}}
\newcommand{\no}{\nonumber}

\usepackage{manyfoot}
\usepackage{amsmath}
\usepackage{amssymb}
\usepackage{amsthm}
\usepackage[mathscr]{eucal}
\usepackage{graphicx}
\usepackage{pstricks}
\usepackage{ifthen}
\usepackage{ulem}
\usepackage{comment}
\usepackage{mathtools}
\usepackage{tikz}
\usepackage{multirow}
\usepackage{ulem}
\usepackage{color,soul}

\graphicspath{{images/}}

\newcommand{\zero}{\boldsymbol{0}}

\renewcommand{\d}{\boldsymbol{d}}

\newcommand{\e}{\boldsymbol{e}}
\newcommand{\f}{\boldsymbol{f}}
\newcommand{\g}{\boldsymbol{g}}
\newcommand{\h}{\boldsymbol{h}}

\newcommand{\n}{\boldsymbol{n}}
\newcommand{\p}{\boldsymbol{p}}
\newcommand{\q}{\boldsymbol{q}}

\renewcommand{\t}{\boldsymbol{t}}
\renewcommand{\u}{\boldsymbol{u}}

\newcommand{\w}{\boldsymbol{w}}
\newcommand{\x}{\boldsymbol{x}}
\newcommand{\y}{\boldsymbol{y}}

\newcommand{\E}{\boldsymbol{E}}

\newcommand{\F}{\boldsymbol{F}}
\newcommand{\G}{\boldsymbol{G}}
\renewcommand{\H}{\boldsymbol{H}}

\renewcommand{\P}{\boldsymbol{P}}

\newcommand{\W}{\boldsymbol{W}}
\newcommand{\Y}{\boldsymbol{Y}}

\newcommand{\bPsi}{\underline{\boldsymbol{\Psi}}}

\renewcommand{\i}{\mathrm{i}}

\newcommand{\CL}{\mathcal{L}}

\newcommand{\CN}{\boldsymbol{\mathcal{N}}}
\newcommand{\CM}{\boldsymbol{\mathcal{M}}}

\newcommand{\bbC}{\mathbb{C}}

\newcommand{\bbN}{\mathbb{N}}

\newcommand{\bbR}{\mathbb{R}}

\newcommand{\nx}{\n(\x)}

\newcommand{\dgdnxpm}{\frac{ \partial G_\pm}{\partial \nx} }

\newtheorem{prop}{Proposition}[section]

\newcommand{\SA}{\underline{\boldsymbol{\mathscr{A}}}}
\newcommand{\SB}{\underline{\boldsymbol{\mathscr{B}}}}
\newcommand{\SC}{\underline{\boldsymbol{\mathscr{C}}}}
\newcommand{\SD}{\underline{\boldsymbol{\mathscr{D}}}}
\newcommand{\SE}{\underline{\boldsymbol{\mathscr{E}}}}

\newcommand{\SF}{{{\mathscr{F}}}}
\newcommand{\SG}{{{\mathscr{G}}}}
\newcommand{\SH}{{{\mathscr{H}}}}

\newcommand{\SK}{{{\mathscr{K}}}}

\newcommand{\SSS}{{{\mathscr{S}}}}

\newcommand{\CEE}{{\boldsymbol{\mathcal{E}}}}
\newcommand{\CHH}{{\boldsymbol{\mathcal{H}}}}

\newcommand{\half}{\frac{1}{2}}

\newcommand{\cross}{\times}
\newcommand{\tendsto}{\rightarrow}
\renewcommand{\to}{\rightarrow}

\newcommand{\curl}{\boldsymbol{\mathrm{curl}} \; }

\newcommand{\curlx}{\boldsymbol{\mathrm{curl}}_{\x} }

\newcommand{\gradx}{\boldsymbol{\mathrm{grad}}_{\x}  }

\newcommand{\real}{\mathrm{Re}}

\renewcommand{\sl}{\CL}

\newcommand{\xh}{\widehat{\x}}
\newcommand{\yh}{\widehat{\y}}

\newcommand{\surface}{\partial D}
\newcommand{\sphere}{S}

\newcommand{\poly}{1.584}
\newcommand{\mishi}{1.5 + 0.02i}
\newcommand{\mishii}{1.0925 + 0.248i}

\newcommand{\errorsuffix}{}

\newcommand{\extfieldfig}[9]{\begin{figure}
    \centering
    \revisedB
    \fbox{%
    \begin{tabular}{cc}
        \includegraphics[width=7.5cm]{#5} & \includegraphics[width=7.5cm]{#6}\\
        \includegraphics[width=7.5cm]{#5\errorsuffix} & \includegraphics[width=7.5cm]{#6\errorsuffix}
    \end{tabular}}
    \caption{\label{#7}
      Visualization of the interior and exterior total electric fields
      $| \mathop{\mathrm{Re}} \E_{\tot,n}|$ 
      induced by a H-polarized plane wave
  striking a {#1} with refractive index $\nu = {#2}$
  and electromagnetic size $s = {#3}$ (top left) and $s={#4}$ (top right)
  computed with $n={#8}$ and $n={#9}$ respectively.
  The corresponding relative error
  $\log_{10}(\mathrm{ERR}_{\pm,n}(\x))$ 
  is visualized (bottom left and right).
    }
\end{figure}}

\newcommand{\farfieldfig}[7]{\begin{figure}
    \centering
    \begin{tabular}{cc}
        \includegraphics[width=6.5cm]{#4} & \includegraphics[width=6.5cm]{#5}
    \end{tabular}
    \caption{\label{#6}
      Visualization of the radar cross section (HH-polarization)
      computed with $n={#7}$
  of a {#1} with refractive index $\nu = {#2}$
  and electromagnetic size $s = {#3}$
  in original orientation (blue) and $\pi/2$-rotated orientation (red).
  Coincidence of the curves at $\theta = \pi/2$ is marked by a $\circ$.
  The right figure shows a close-up of the region indicated by the box
  in the left figure.
    }
\end{figure}}

\newcommand{\mishfig}[8]{\begin{figure}
    \centering
    \revisedC
    \fbox{%
    \begin{tabular}{cc}
      \includegraphics[width=6.5cm]{#5} & \includegraphics[width=6.5cm]{#6}
    \end{tabular}}
    \caption{\label{#7}
      Visualization of the radar cross section (HH-polarization)
      computed with $n={#8}$
      of a spheroid with aspect ratio #1, refractive index
      $\nu = {#2}$ (left) and $\nu={#3}$ (right),
  and electromagnetic size $s = {#4}$ computed using
  this algorithm (blue) and Mishchenko's code (red dashed).
    }
\end{figure}}

\newcommand{\tablecap}{}

\newenvironment{rptable}[3]{%
    \renewcommand{\tablecap}{\caption{\label{#3}
      Scattering by a {\bf {#1}} with diameter {#2}.}}
  \begin{table}
    \centering
\begin{tabular}{cccc}
  \hline
   & ~~~~$\nu = 1.584$~~~~ & $\nu = 1.5 + 0.02i$ & $\nu = 1.0925 + 0.248i$\\
  \cline{2-4}
  $n$ & $\mathrm{ERR}_{\mathrm{RP},n}$ & $\mathrm{ERR}_{\mathrm{RP},n}$ & $\mathrm{ERR}_{\mathrm{RP},n}$ \\
  \hline}
{%
\hline
\end{tabular}
\tablecap
  \end{table}
}

\newenvironment{mietable}[3][]{%
    \renewcommand{\tablecap}{\caption{\label{#3}
      Scattering by a {\bf sphere} with diameter {#2}.}}
    \begin{table} #1
    \centering
\begin{tabular}{cccc}
  \hline
   & ~~~~$\nu = 1.584$~~~~ & $\nu = 1.5 + 0.02i$ & $\nu = 1.0925 + 0.248i$\\
  \cline{2-4}
  $n$ & $\mathrm{ERR}_{\mathrm{Mie},n}$ & $\mathrm{ERR}_{\mathrm{Mie},n}$ & $\mathrm{ERR}_{\mathrm{Mie},n}$ \\
  \hline}
{%
\hline
\end{tabular}
\tablecap
  \end{table}
}

\newenvironment{eddytable}[3][]{%
    \renewcommand{\tablecap}{\caption{\label{#3}
      Scattering by a plasmonic {\bf sphere} with diameter {#2}.}}
    \begin{table} #1
    \centering
\begin{tabular}{cccc}
  \hline
   & ~~~~$\nu = 10^2(1+i)$~~~~ & $\nu = 10^3(1+i)$ & $\nu = 10^4(1+i)$\\
  \cline{2-4}
  $n$ & $\mathrm{ERR}_{\mathrm{Mie},n}$ & $\mathrm{ERR}_{\mathrm{Mie},n}$ & $\mathrm{ERR}_{\mathrm{Mie},n}$ \\
  \hline}
{%
\hline
\end{tabular}
\tablecap
  \end{table}
}

\newtheorem{theorem}{Theorem}
\newtheorem{lemma}[theorem]{Lemma}
\newtheorem{remark}[theorem]{Remark}

\newcommand{\gb}{\boldsymbol{\gamma}}
\newcommand{\gt}{\boldsymbol{\gamma}_{\boldsymbol{t}}}
\newcommand{\gn}{\gamma_{\boldsymbol{n}}}

\newcommand{\tot}{\mathrm{tot}}
\newcommand{\inc}{\mathrm{inc}}

\newcommand{\Mr}{\mathbb{M}}

\newcommand{\Mm}{\boldsymbol{\Mr}}

\newcommand{\Ir}{\mathbb{I}}
\newcommand{\Ima}{\boldsymbol{\Ir}}
\newcommand{\Jr}{\mathbb{J}}
\newcommand{\Jm}{\boldsymbol{\Jr}}

\newcommand{\Cs}{\mathrm{C}}
\newcommand{\Cv}{\boldsymbol{\Cs}}
\newcommand{\Hs}{\mathrm{H}}
\newcommand{\Hv}{\boldsymbol{\Hs}}
\newcommand{\Hss}{\Hs^s}
\newcommand{\Hvs}{\Hv^s}

\newcommand{\Ls}{\mathrm{L}}
\newcommand{\Lv}{\boldsymbol{\Ls}}

\definecolor{marker}{rgb}{1,1,0.5}
\definecolor{marker}{rgb}{1,1,0.5}
\definecolor{mydarkgreen}{rgb}{0,0.6,0}
\definecolor{paleyellow}{rgb}{1,1,0.9}

\newcommand{\revA}[1]{\textcolor{black}{#1}}
\newcommand{\revB}[1]{\textcolor{black}{#1}}
\newcommand{\revC}[1]{\textcolor{black}{#1}}
\newcommand{\us}[1]{\textcolor{black}{#1}}

\newcommand{\revisedB}{\color{black}}
\newcommand{\revisedC}{\color{black}}
\newcommand{\revisedus}{\color{black}}

\begin{document}

\title{
{\revB{ 
An  all-frequency stable  integral system
for Maxwell's equations in  3-D penetrable media:
continuous and discrete  model analysis}}}


\author*[1]{\fnm{Mahadevan} \sur{Ganesh}}\email{mganesh@mines.edu}

\author[2]{\fnm{Stuart C.} \sur{Hawkins}}\email{stuart.hawkins@mq.edu.au}

\author[3]{\fnm{Darko} \sur{Volkov}}\email{darko@wpi.edu}

  \affil*[1]{\orgdiv{Department of Applied Mathematics and Statistics},
    \orgname{Colorado School of Mines}, \orgaddress{\city{Golden}, \state{CO},
    \country{USA}}}
  
  \affil[2]{\orgdiv{Department of Mathematical and Physical Sciences},
    \orgname{Macquarie University}, \orgaddress{\city{Sydney}, \state{NSW} \postcode{2109},
    \country{Australia}}}

    \affil[3]{\orgdiv{Department of Mathematical Sciences},
    \orgname{Worcester Polytechnic Institute}, \orgaddress{\city{Worcester}, \state{MA},
    \country{USA}}}

    \abstract{		    
 {\revB{We introduce a new system of surface integral equations for Maxwell's transmission problem in three dimensions (3-D). This system has two remarkable features, both of which we prove. First, it is well-posed at all frequencies. Second, the underlying linear operator has a uniformly bounded inverse as the frequency approaches zero,  ensuring that there is no low-frequency breakdown.
The system is derived from a formulation we introduced in our previous work, which required additional integral constraints to ensure well-posedness across all frequencies. In this study, we eliminate those constraints and demonstrate that our new self-adjoint, constraints-free linear system—expressed in the desirable form of an identity plus a compact weakly-singular operator—is stable for all frequencies.
Furthermore, we propose and analyze a fully discrete numerical method for these systems and provide a proof of spectrally accurate convergence for the computational method. We also computationally demonstrate the high-order accuracy of the algorithm using  benchmark scatterers with curved surfaces.
		}}}


\maketitle

\section{Introduction}\label{sec:int}

Computation of the radar cross section (RCS) through simulation of electromagnetic waves in unbounded dielectric 
or absorbing penetrable media has been, and remains, crucial for advances in several technological and atmospheric 
science innovations~\cite{cheney,hulst:light, kristensson,rother}.
The classical approach to modeling the electric and magnetic fields 
in the presence of a penetrable body in free space is based on the 
  Maxwell partial differential equations (PDEs), augmented with 
a radiation condition for the exterior fields,
and tangential continuity of the 
electric field and magnetizing field
across the  interface~\cite{costabel:2011_pap,colton:inverse, nedlec:book}.

{\revB{The Maxwell PDE system for the dielectric case has been proven to be uniquely solvable at all frequencies~\cite{muller1946, muller:book}. 
As noted in the introduction of \cite{buffa2000justification}, this PDE system remains uniformly well-posed as the frequency approaches zero. 
This uniform well-posedness holds for the dielectric case, regardless of the scatterer's genus. 
In contrast, for perfect conductors, additional conditions must be verified, as explained in \cite{buffa2000justification}. 
This observation is consistent with the findings of the present authors in \cite[Appendix B]{ganesh2014all}. }}

Our focus in this work is on time-harmonic
electromagnetic waves with positive frequency $\omega$, 
characterized   in a 
three dimensional  piecewise-constant medium
by the electric field $\CEE$ 
and magnetic field $\CHH$~\cite[Page~178]{nedlec:book} with
\begin{equation}
\label{eq:timewave}
\CEE(\x,t) =  \real \left\{\E(\x) e^{-i \omega t} \right\}, \quad 
\CHH(\x,t)  =    \real \left\{\H(\x) e^{-i \omega t} \right\}. 
\end{equation}
\revB{We denote by $\surface$  the boundary of a closed bounded region $D \subset \bbR^3$,
and write $D^+ = \bbR^3 \setminus \overline{D}$ and $D^- = D$
for the exterior and
interior of the penetrable domain $D$, respectively. The regularity of the electromagnetic field on the surface depends
on the appropriate smoothness of the surface $\surface$. We quantity such smoothness requirements in the first theorem of this article.}

We denote the 
permittivity and permeability constants
in   $D^\pm$
by $\epsilon^\pm$ and $\mu^\pm$ respectively.
{\revA{For the mathematical and numerical analysis of our  new all-frequency stable formulation, we focus on non-plasmonic scatterers $D$. For such physical materials, following the literature~\cite{colton:inverse, nedlec:book}, we assume that $\mu^\pm$ and $\epsilon^+$ are positive real numbers, and non-zero $\epsilon^-$ may be either a positive real number or a complex number 
    with non-negative real and imaginary parts.  Although our analyses may not apply to all cases involving real negative permittivity
    (such as the plasmonic case), the formulation and  algorithm we present can still be utilized for a subset of plasmonic scenarios, which we will demonstrate in the last section of this article. In future work, we plan to develop mathematical and numerical analysis specifically for plasmonic materials.}}

It is convenient to then define
\begin{equation}
\label{eq:epsilon-mu}
\epsilon(\x) = \epsilon^\pm, \qquad \mu(\x) = \mu^\pm, \qquad \mbox{for $\x \in D^\pm$}.
\end{equation}
The refractive index of the medium is
\begin{displaymath}
  \nu = \left( \frac{\mu^- \epsilon^-}{\mu^+ \epsilon^+} \right)^{1/2},
\end{displaymath}
and the penetrable media is called dielectric when $\nu$ is real,
and absorbing otherwise (in which case it follows from the assumptions
above that the imaginary part of $\nu$
is positive).

Interaction of an incident electromagnetic field with the penetrable
body $D$ induces an interior field with spatial components $[\E^-,\H^-]$ in $D^-$ and
a scattered field with spatial components $[\E^+,\H^+]$ in $D^+$.
The spatial components $[\E,\H]$ of the induced field, with
\begin{equation}\label{eq:int_ext_rep}
\E(\x) = \E^\pm(\x), \qquad \H(\x) = \H^\pm(\x), \qquad \qquad  \x \in D^\pm,
\end{equation}
satisfy the time-harmonic Maxwell equations~\cite[Page~253]{nedlec:book}
\begin{equation}
\label{eq:reduced-maxwell}
\curl \E(\x) - \i \omega \mu(\x) \H(\x)  =  \zero,  \qquad 
\curl \H (\x) + \i\omega \epsilon(\x) \E(\x)  =  \zero,\qquad  \x \in  D^\pm,
\end{equation}
and the Silver-M\"uller radiation condition
\begin{equation}
\label{eq:silver-muller}
\lim_{|\x| \tendsto \infty} \left[\sqrt{\mu^+}  \H(\x) \cross \x -\sqrt{\epsilon^+} | \x | \E(\x) \right] = \zero.
\end{equation}
The exterior and interior  fields  are induced by an incident electromagnetic
field  $[\E_{\inc},\H_{\inc}]$ impinging on the penetrable body. 
We require  $[\E_{\inc},\H_{\inc}]$   to satisfy the homogeneous
Maxwell equations
\begin{equation}
\label{eq:incident-maxwell}
\curl \E_\inc(\x) - \i \omega \mu^+ \H_\inc(\x)  =  \zero,  \quad 
\curl \H_\inc (\x) + \i\omega \epsilon^+ \E_\inc(\x)  =  \zero,\quad  \x \in  \bbR^3 \setminus Q,
\end{equation}
where $Q \subset \bbR^3$ is a compact or empty set, bounded away from the 
penetrable body $D$. 
In applications the incident wave is typically a
plane wave, or a wave induced by a point source that satisfies~\eqref{eq:incident-maxwell}.
The total field is therefore $[\E_\tot,\H_\tot]$ with
\begin{equation}\label{eq:total_field}
\E_\tot(\x) = \E_\tot^\pm(\x), \qquad \H_\tot(\x) = \H_\tot^\pm(\x), \qquad \qquad  \x \in D^\pm,
\end{equation}
with
\begin{equation}
\E_\tot^- = \E^-, \quad \E_\tot^+ = \E^+ + \E_\inc, \quad
\H_\tot^- = \H^-, \quad \H_\tot^+ = \H^+ + \H_\inc.
\end{equation}
We highlight that inside $D$ the total field is the induced field.


The tangential components of the total electric and magnetic fields 
are required to be  continuous across the interface $\surface$, leading
to the interface conditions~\cite[Equation~(5.6.66), Page~234]{nedlec:book}
\begin{equation}\label{eq:EH_interface}
 \E^-_\tot \times \n =     \E^+_\tot \times \n,  
\qquad \H^-_\tot \times \n  =    \H^+_\tot \times \n,
\qquad \text{on $\surface$}, 
\end{equation}
where $\n$ is the outward unit normal to $\surface$. 
A consequence of~\eqref{eq:reduced-maxwell},~\eqref{eq:incident-maxwell}, 
and~\eqref{eq:EH_interface}
is that
the normal components of the fields (called the surface charges) also satisfy  interface conditions
\begin{equation}\label{eq:nor_interface}
\epsilon^- \n \cdot \E_\tot^- =  \epsilon^+ \n \cdot \E_\tot^+, \qquad
\mu^- \n \cdot \H_\tot^- =  \mu^+ \n \cdot \H_\tot^+, \qquad \text{on $\surface$}. 
\end{equation}
The scattered field $[\E^+,\H^+]$ has the asymptotic behavior of
an outgoing spherical wave~\cite[Theorem~6.8, Page~164]{colton:inverse},
and in particular,
\begin{equation}
\label{eq:farfield}
\E^+(\x) = \frac{e^{\i k_+ |\x|}}{|\x|} \left\{ \E^\infty(\xh) + O \left( \frac{1}{|\x|}\right)    \right\},
\end{equation}
as $|\x| \to \infty$ uniformly in all directions $\xh = \x/|\x|$,
where $k_+ = \omega \sqrt{\mu^+ \epsilon^+}$ is the exterior wavenumber.
The vector field $\E^\infty$ is known as the far field
patterns of $\E^+$.

One way to solve the Maxwell PDE system given by 
(\ref{eq:reduced-maxwell})--(\ref{eq:nor_interface}) is to formulate
equivalent surface integral equations (SIEs) on 
$\surface$.  {\revA{ A key reason for reformulating the Maxwell PDEs as an SIE system is to facilitate the solution of the {\it unbounded} region model (\ref{eq:reduced-maxwell})–(\ref{eq:nor_interface}), through an equivalent model defined only on the {\it bounded} surface  without truncation of the unbounded region and also  satisfy the radiation condition at infinity exactly.  In addition to the computational advantages of solving the 3-D Maxwell transmission problem
    using an SIE—where we only need to discretize for unknowns on the two-dimensional bounded surface—the surface integral operators in the SIE exhibit desirable boundedness and compactness properties compared to the unbounded differential operators in the PDEs. 
The foundational SIEs for non-spherical particles, known as the Electric Field Integral Equation (EFIE) and the Magnetic Field Integral Equation (MFIE), were proposed in Maue's work from 1949~\cite{Maue}. The issue of spurious resonances in these integral equations was addressed two decades later using a formulation known as the Combined Field Integral Equation (CFIE)~\cite{nedlec:book}. The low-frequency breakdown problem in such SIEs has only been tackled in this century. 
In particular, for the dielectric model, all-frequency stable continuous and discrete formulations—subject to additional constraints—were first developed in~\cite{eps_gre:2009_pap,eps_gre:2011_pap}, which required the inversion of the surface Laplace–Beltrami differential operator. Subsequently, the authors presented further developments in~\cite{ganesh2014all,dielectric2} with constraints on integral operators. However, those articles do not provide numerical analysis to mathematically quantify appropriate validity of the proposed discrete formulations.}}

In previous work~\cite{ganesh2014all} we discussed various classes
 of SIE for the Maxwell dielectric problem
 and
 introduced our own original formulation which was an integral equation
 combined with a stabilization constraint.
We showed that this formulation possesses
a few remarkable properties that, taken together,
are unique: first, this system is in the classical form of the
identity plus a compact operator. 
Second, it does not suffer from spurious eigenfrequencies. Third, 
it does not suffer from the low-frequency breakdown phenomenon: as the frequency tends to zero, this system remains uniformly well posed. 
Fourth, the solution to combined system is equal to the trace of the electric and magnetic fields
on the interface.

{\revB{All these four points were established
in~\cite{ganesh2014all}, subject to the non-plasmonic regime parameters assumptions described earlier. For the entire plasmonic regime, 
we are not aware of any SIE formulation in the literature  that has been proven or even numerically demonstrated to exhibit all  these four remarkable 
features. In the case of the plasmonic regime, as we numerically demonstrate in~Section~\ref{sec:numerics}, and as also shown in~\cite{helsing_plasmon} for several other SIE formulations,
it is plausible that  these four features may hold only for a subset of parameters within the plasmonic region. We plan to explore and quantify such a subset region through mathematical and numerical analysis
in  future work.}}

The main focus of this article, {\us{for the non-plasmonic region}},  is to implement the stabilization derived in~\cite{ganesh2014all}  in such a way
that a high-order numerical
method can be developed, and prove spectral convergence of the numerical
algorithm.
To this end, one of the main issues is how to incorporate the stabilization
into the integral equation formulation and retain an all-frequency stable
uniquely-solvable system.
One of the established approaches, as in the traditional {\us{CFIE
(combining the EFIE and MFIE)}} is to take a linear combination of the
integral equation and the stabilization constraint.
Such an approach has been used in the engineering literature~\cite{Yla_Tas:di_low,Tas_Van:pic_low}
without any proof that the uniquely solvable property is retained for
all frequencies.
We used such an approach for preliminary numerical investigations in~\cite{ganesh2014all}
and demonstrated that, for spherical scatterers,
there are some coupling parameters for which
the unique solvability cannot be guaranteed for all
frequencies.
One of the key contributions in this present work is to establish, using
robust mathematical proofs, that such a simple one-parameter coupling approach
will lead to a non-uniquely solvable system for infinitely many choices of
the parameters.
Accordingly, using such a simple approach does not guarantee, in general,
the
key properties of the all-frequency stable integral equation formulation,
subject to the constraint.

Subsequently, in this article we introduce a new continuous formulation that
incorporates the constraint efficiently into one equation, and we prove that
the established unique solution, with the constraint, is equivalent to
the solution obtained using our new coupled formulation.
The main focus of this paper is to apply a high-order numerical algorithm
for the new symmetric formulation, and prove spectrally accurate convergence under certain conditions that
are especially suited for curved interfaces. 
We substantiate our robust mathematical results with numerical experiments.
The underlying spectrally-accurate numerical framework was developed
for perfect conductors in~\cite{gh:first} and subsequently further developed
for the dielectric model in~\cite{dielectric2}, where we incorporated
the constraint using a complicated saddle-point formulation.
Numerical analysis of such a saddle-point approach is yet to be explored.
The present article is the first work in the literature  to develop  numerical analysis for
an all-frequency stable and uniquely solvable weakly-singular SIE system.

The rest of this article is as follows:
\begin{itemize}
\item[--]
In Section~\ref{sec:oper} we first review 
properties of surface integral operators that are relevant to this
study.
\item[--]
We then review an all-frequency stable surface
integral equation subject to  a stabilization constraint, and a main result from~\cite{ganesh2014all} in Section~\ref{sec:max_to_sie}.
\item[--]
In Appendix~\ref{Appendix1} we prove that the simple approach of combining
the integral equation with the stabilization constraint
using a single parameter does not, in general, guarantee unique solvability
of the resulting system, for infinitely many choices of the parameters.
\item[--]
Motivated by this, 
in Section~\ref{sec:J} we introduce a new equivalent formulation
incorporating the stabilization constraint in a self-adjoint form
and prove that the new single-equation system, without any constraint,
is uniquely solvable.
\item[--]
In Section~\ref{sec:spectral} we introduce a fully-discrete numerical
scheme for our new single-equation weakly-singular SIE system and prove that the associated
numerical solution converges super-algebraically.
\item[--]
In Section~\ref{sec:numerics} we implement the spectrally accurate
algorithm and demonstrate the superalgebraic convergence
for a class of convex and non-convex
test obstacles that are widely used in the literature.
\end{itemize}

\section{Surface integral operators and properties}\label{sec:oper}

In this section we introduce the  integral operators that 
subsequently appear in our SIE.  
The starting point is the use of particular potentials
and their traces onto the interface $\surface$.
We use the same notation as in our previous article \cite{ganesh2014all}
and we 
first recall from \cite{ganesh2014all} some notation and definitions.
For related derivations and proofs, we refer to \cite{ganesh2014all}.

Let $\Omega \subseteq \bbR^3$ be either  $D^-$, $D^+$, 
$\surface$ or  $\bbR^3$. 
For a non-negative integer $m$  and $\sigma \in \bbR$ with $0 < \sigma \leq 1$,
let $\Cs^m(\overline{\Omega})$ 
denote the standard space of $m$-times continuously differentiable functions
and   $\Cs^{m,\sigma}(\overline{\Omega})$ denote the   H\"older space of all functions  in
$\Cs^m(\overline{\Omega})$ whose derivatives of order $m$ are uniformly H\"older continuous with
exponent $\sigma$. 
For $s\in \bbR$ and bounded $\Omega$ the space $\Hss(\Omega)$ is  the standard Sobolev space of
order $s$.

The  vector-valued counterparts of these 
spaces will be denoted by boldface symbols. Thus, for $m \in \bbN, \; s \in \bbR, 0\leq \sigma \leq 1$,
\begin{equation}
\Cv^{m,\sigma}(\overline{\Omega}) = \Big[ \Cs^{m,\sigma}(\overline{\Omega}) \Big]^3,
\quad \Hvs(\Omega) = \Big [ \Hss(\Omega) \Big]^3. 
\end{equation}

Let $\gb^\pm$ be the standard exterior/interior trace operator defined
for a continuous vector field $\u$  on $\overline{D^\pm}$ as
\begin{equation}\label{eq:lim_def}
\Big(\gb^\pm \u\Big)(\x) =  
\lim_{h \to 0^{\pm}} \u(\x+h\n(\x)), \qquad \x \in \surface.
\end{equation}
Let $\gt^\pm, \gn^\pm$ be respectively  the  (twisted) tangential  and  normal 
trace operators,  defined for a continuous vector field $\u$  on $\overline{D^\pm}$ as
\begin{equation}\label{eq:trace_def}
 \gt^\pm \u = (\gb^\pm\u) \times \n.  \qquad  \gn^\pm \u = (\gb^\pm \u) \cdot \n. 
\end{equation}
We denote by 
$\Cv_{\t}^{m,\sigma}(\surface)$ and $ \Hvs_{\t}(\surface)$ respectively
the spaces of tangential vector fields
 on $\surface$ with regularity $\Cv^{m,\sigma}$ and $\Hvs$.

Let
\begin{equation}\label{eq:fund}
G_\pm(\x,\y)  = \frac{1}{4 \pi} \frac{ e^{\i k_\pm |\x - \y|} }{|\x - \y|},
\end{equation}
where $k_\pm = \omega \sqrt{\epsilon^\pm \mu^\pm}$ denotes the exterior  and
interior wavenumber respectively. 
Let $\Psi^V_\pm$ and $\bPsi^{\boldsymbol{V}}_\pm$ be respectively the scalar and vector single layer
potential operators, defined for scalar and vector surface densities $w$ and $\w$ as
\begin{equation}\label{eq:sl_pot}
(\Psi^V_\pm w)(\x)  =  \int_{\surface} G_\pm(\x,\y)  \; w(\y) \; ds(\y), \quad
(\bPsi^{\boldsymbol{V}}_\pm \w)(\x)  =  \int_{\surface} G_\pm(\x,\y)  \; \w(\y) \; ds(\y), 
\end{equation}
for $\x \in \bbR^3 \setminus \surface$.
We assume that the surface $\surface$ is at least of class $\Cs^{m^*, 1}$, for some integer $m^* \geq 1$.
Then the operators
$\Psi^V_\pm: \Cs(\surface) \rightarrow \Cs^{0,\alpha}(\bbR^3)$
and, for $0 < \alpha  < 1$, $\bPsi^V_\pm: \Cv(\surface) \rightarrow \Cv^{0,\alpha}(\bbR^3)$ 
are bounded linear operators~\cite[Theorems~3.3, 6.12]{colton:inverse}.
Hence
$\gt^\pm \circ \bPsi^{\boldsymbol{V}}_\pm :  \Cv(\surface) \rightarrow  
\Cv_{\t}^{0,\alpha}(\surface)$
is a bounded linear operator  and consequently   compact as an operator onto $\Cv(\surface)$
because of the compact imbedding~\cite[Theorem~3.2]{colton:inverse}. 
Similarly, using the properties of the  single layer potential on Sobolev 
spaces (see~\cite[Theorem 1(i)]{costabel:88_pap}), 
$\gt^\pm \circ \bPsi^{\boldsymbol{V}}_\pm$  is a compact operator onto  $\Hv^\half(\surface)$. 
Using the jump relation for the single layer 
potential~\cite{colton:inverse, nedlec:book}, we obtain
\begin{equation}\label{eq:lim0}
\gt^\pm \circ \bPsi^{\boldsymbol{V}}_\pm \w = \SA_\pm \w, \qquad \gn^\pm  \circ\bPsi^{\boldsymbol{V}}_\pm \w = \SF_\pm \w,
\end{equation}
where the tangential-field-valued and scalar-valued operators are defined as 
\begin{eqnarray}
(\SA_\pm \w)(\x)  & =  &  \int_{\surface} G_\pm(\x,\y)  \; \w(\y) \times \n(\x) \; ds(\y), \qquad \x \in \surface, \label{eq:A} \\
(\SF_\pm \w)(\x)  & = &   \int_{\surface} G_\pm(\x,\y)  \; \w(\y)\cdot \n(\x) \; ds(\y), \qquad \x \in \surface.
\label{eq:F} 
\end{eqnarray}
It is useful to consider the normal-field operator
\begin{equation}
(\n \SF_\pm \w)(\x)   =   \n(\x) (\SF_\pm \w)(\x), \qquad \x \in \surface \label{eq:nF}.
\end{equation}
The  surface integral operators 
$\SA_\pm$  and $\n \SF_\pm$  are compact as operators on  $\Cv(\surface)$ and $\Hv^\half(\surface)$. 
Further, for $\alpha$, $s$ and integer $m$ satisfying
\begin{equation}
  \label{eq:conditions}
  0 < \alpha < 1, \qquad 0 \leq m < m^*, \qquad 0 < s \leq m^*,
\end{equation}
we have that
\begin{equation}  \label{cont of A}
\SA_\pm,\n\SF_\pm :  \Cv^{m}(\surface) \rightarrow  \Cv^{m, \alpha}(\surface), \qquad  \qquad 
\SA_\pm,\n\SF_\pm :  \Hv^{s-1}(\surface) \rightarrow  \Hv^{s}(\surface),
\end{equation}
are bounded linear operators. The operators in \eqref{eq:A}--\eqref{eq:nF} enjoy
these  smoothing (and hence compactness) properties essentially because the kernels of these
surface integral operators are weakly-singular. 

We define vector-valued surface integral operators
\begin{eqnarray}
(\SB_\pm w)(\x) & = & \int_{\surface} \gradx G_\pm(\x,\y)  \times \n(\x) \; w(\y) \; ds(\y),  \label{eq:B}\\
(\SC_\pm \w)(\x) & = & 
- \n(\x) \times \curlx \int_{\partial D} G_\pm(\x,\y) \w(\y) \; ds(\y), \label{eq:C1}
\end{eqnarray}
for $\x \in \surface$, and related scalar-valued surface integral operators 
\begin{eqnarray}
(\SG_\pm w)(\x) & = & \int_{\surface} \dgdnxpm(\x,\y)  \; w(\y) \; ds(\y), \label{eq:G} \\
(\SH_\pm \w)(\x) & = & \int_{\surface}  \Big( \gradx G_\pm(\x,\y) \times \n(\x) \Big) \cdot \w(\y) \; ds(\y).
\label{eq:H}
\end{eqnarray}
We note that for a tangential density $\w$, 
\begin{equation}
(\SC_\pm \w)(\x)     =  (\SD_\pm \w)(\x) - (\SE_\pm \w)(\x),   \label{eq:C} 
\end{equation}
where for $\x \in \surface$, 

\vspace{-0.3in}
\begin{eqnarray}
(\SD_\pm \w)(\x) & = & \int_{\surface} \dgdnxpm(\x,\y)  \; \w(\y) \; ds(\y),  \label{eq:D} \\
(\SE_\pm \w)(\x) & = &  \int_{\surface} \Big( \gradx G_\pm(\x,\y) \Big) \; \w(\y) \cdot \Big( \n(\x) - \n(\y) \Big) \; ds(\y).   \label{eq:E}
\end{eqnarray}
It is also useful to consider the normal field surface integral operators
\begin{equation}
\label{eq:gh}
(\n \SG_\pm w)(\x) = \n(\x) (\SG_\pm w)(\x), \qquad  (\n \SH_\pm \w)(\x) = \n(\x) (\SH_\pm \w)(\x),
\qquad \x \in \surface.
\end{equation}

Clearly, the  kernels of the surface integral operators $\SD_\pm$ and $\SG_\pm$ are weakly-singular.
Hence $\SG_\pm$ is a compact  operator on $\Cs(\surface)$
and $\Hs^\half(\surface)$. 

Since $\surface$ is of at least $\Cs^{1, 1}$-class, we obtain
\begin{equation}
| \n(\x) - \n(\y) | \leq  C | \x - \y |, \qquad \x, \y \in \surface,
\end{equation}
for some constant $C$ that depends on the geometry of $\surface$.
Thus, because of the term $\n(\x) - \n(\y)$ in \eqref{eq:E}, 
 the operator  $\SE_\pm$ is weakly-singular. 
{\revA{There is a vast amount of literature
\cite{colton2013integral, maz1991boundary, nedelec2001acoustic}
explaining  that this implies that
$\SC_\pm$  is a compact  operator on $\Cv_{\t}(\surface)$
and on $\Hv^\half_{\t}(\surface)$.}}
In addition,
if~\eqref{eq:conditions} holds then
\begin{eqnarray*}
& & \SC_\pm:  \Cv^{m}_{\t}(\surface) \rightarrow  \Cv^{m, \alpha}_{\t}(\surface), \qquad 
\SC_\pm:  \Hv^{s-1}_{\t}(\surface) \rightarrow  \Hv^{s}_{\t}(\surface), \nonumber \\
& & \SG_\pm :  \Cs^{m}(\surface) \rightarrow  \Cs^{m, \alpha}(\surface), \qquad 
\SG_\pm :  \Hs^{s-1}(\surface) \rightarrow  \Hs^{s}(\surface),
\end{eqnarray*}
are bounded linear operators. 
The kernels of the surface integral operators $\SB_\pm,  \SH_\pm$  in
\eqref{eq:B}, \eqref{eq:H} have order $1/r^2$ singularity. 
However, our  reformulation of the time-harmonic dielectric model 
requires only operators $\SB_+ - \SB_-$
and $\SH_+ - \SH_-$.  It follows from~\cite[Remark 3.1.3]{sau_sch:book} 
that $\SB_+ - \SB_-$ and  $\SH_+ - \SH_-$ have weakly-singular kernels.
Thus the operator  $\n\SH_+  -  \n\SH_-$ is
compact on $\Cv(\surface)$
and $\Hv^\half(\surface)$ whilst
the operator  $\SB_+  - \SB_-$ is
compact on  $\Cs(\surface)$
and $\Hs^\half(\surface)$, and
if~\eqref{eq:conditions} holds then  the following linear operators are bounded:
\begin{eqnarray*}
& & \n\SH_+  -  \n\SH_- : 
\Cv^{m}(\surface) \rightarrow  \Cv^{m, \alpha}(\surface), \qquad
\n\SH_+  -  \n\SH_- :  \Hv^{s-1}(\surface) \rightarrow  \Hv^{s}(\surface), \\
& & \SB_+  - \SB_- :  \Cs^{m}(\surface) \rightarrow  \Cv^{m, \alpha}(\surface), \qquad \qquad
 \SB_+  - \SB_- :  \Hs^{s-1}(\surface) \rightarrow  \Hv^{s}(\surface). 
\end{eqnarray*}

Furthermore, we define two further scalar integral operators that we use in our additional
scalar integral equations for indirectly imposing the identities connecting 
currents and charges,
\begin{eqnarray}
(\SSS_\pm w)(\x) & = & \int_{\surface} G_\pm(\x,\y)  \; w(\y) \; ds(\y),  \quad \x \in  \surface, \label{eq:S} \\
(\SK_\pm \w)(\x) & = & \int_{\surface}  \Big( \gradx G_\pm(\x,\y) \times \n(\y) \Big) \cdot \w(\y) \; ds(\y), 
\quad \x \in \surface.
\label{eq:K}
\end{eqnarray}
The operator  $\SSS_+-  \SSS_- $ is
compact on $\Cs(\surface)$
and $\Hs^\half(\surface)$ whilst
the operator  $\SK_+  - \SK_-$ is
compact on  $\Cv(\surface)$
and $\Hv^\half(\surface)$, and
if~\eqref{eq:conditions} holds then
the following linear operators are bounded:
\begin{eqnarray*}
& & \SSS_+  -  \SSS_- : 
\Cs^{m}(\surface) \rightarrow  \Cs^{m, \alpha}(\surface), \qquad
\SSS_+  -  \SSS_- :  \Hs^{s-1}(\surface) \rightarrow  \Hs^{s}(\surface), \\
& & \SK_+  - \SK_- :  \Cv^{m}(\surface) \rightarrow  \Cs^{m, \alpha}(\surface),
\quad
 \SK_+  - \SK_- :  \Hv^{s-1}(\surface) \rightarrow  \Hs^{s}(\surface).
\end{eqnarray*}
Finally, we introduce
the following weakly-singular operators to
facilitate recalling in the next section our first second-kind surface-integral system that is all-frequency-stable
subject to a stability constraint.  
\begin{eqnarray}\label{eq:sub_operator}
\widetilde{\SC}^{(\alpha^+,\alpha^-)}  &=&  \alpha^+ \SC_+ - \alpha^- \SC_-, \qquad
\widetilde{\SA}^{(\alpha^+,\alpha^-,\beta^+,\beta^-,\lambda)} =
\i \lambda (\beta^+ \alpha^+ \SA_+ - \beta^- \alpha^- \SA_-), 
 \nonumber \\
\widetilde{\SG}^{(\alpha^+,\alpha^-)} &=& \SG_+ - \frac{\alpha^+}{\alpha^-} \SG_-, \qquad
\widetilde{\SF}^{(\alpha^+,\alpha^-,\lambda)} = \i \lambda \left(\alpha^+ \SF_+ - \alpha^- \SF_-\right),
  \nonumber \\
 \widetilde{\SB} = \SB_+ &- &  \SB_-,
\qquad \widetilde{\SH} = \SH_+ -\SH_-, 
\quad 
  \widetilde{\SSS} =  (\SSS_+ - \SSS_-), \quad
   \widetilde{\SK}  =  (\SK_+ - \SK_-).
\end{eqnarray}
In this notation, when it is used subsequently, 
 the parameter $\lambda$ will be replaced by $\pm \omega$ and hence 
 $\widetilde{\SA}^{(\alpha^+,\alpha^-,\beta^+,\beta^-,\pm \omega)}$ and $\widetilde{\SF}^{(\alpha^+,\alpha^-,\pm \omega)}$
 tend to zero operators as $\omega \rightarrow 0$. 

\section{A stability constrained well-posed SIE system}\label{sec:max_to_sie}

In this section we recall from~\cite{ganesh2014all} a matrix-operator 
 and  a stability constraint, and hence, 
subject to the constraint, an all-frequency-stable weakly-singular SIE system that is equivalent to the Maxwell PDE.
To this  end,
let $\Mm$ be the $2 \times 2$  weakly-singular  operator 
\begin{equation}\label{eq:Mm}
\Mm = \left [ \begin{array}{ll} \CM^{(\epsilon^+,\epsilon^-)}  & \CN^{(\epsilon^+,\epsilon^-,\mu^+,\mu^-,\omega)}   \\ \\
                                   \CN^{(\mu^+,\mu^-,\epsilon^+,\epsilon^-,-\omega)}    &  \CM^{(\mu^+,\mu^-)} 
                     \end{array} \right ], 
\end{equation}
with
\begin{eqnarray}\label{eq:full_Mop}
\CM^{(\alpha^+,\alpha^-)} \w  &=&  \frac{2}{\alpha^+ + \alpha^-}\Bigg \{ \n \times
\left[\widetilde\SC^{(\alpha^+,\alpha^-)}(\w \times \n)+ \alpha^+
{\us{\widetilde\SB}}
(\w \cdot \n)\right] \nonumber \\
 & & \hspace{0.9in} - \alpha^- \n \left [\widetilde\SH (\w \times \n) - \widetilde\SG^{(\alpha^+,\alpha^-)} (\w \cdot \n) \right ] \Bigg \},
\end{eqnarray}
\begin{equation}\label{eq:full_Nop}
\CN^{(\alpha^+,\alpha^-,\beta^+,\beta^-,\lambda)} \w =  
\frac{2}{\alpha^+ + \alpha^-}\left \{ \left [\n \times 
\widetilde\SA^{(\alpha^+,\alpha^-,\beta^+,\beta^-,\lambda)} 
+ \alpha^- \n \widetilde\SF^{(\beta_+,\beta_-,\lambda)}\right] (\w \times \n) \right\}.
\end{equation}
It is easy to see the important property of $\Mm$ that in the limit $\omega \rightarrow 0$, the off-diagonal entries of $\Mm$ are zero, and 
hence the resulting diagonal operator facilitates decoupling of the electric and magnetic fields as $\omega \rightarrow 0$.
We further  observe that, because of the weak singularity of their kernels,  $\CM^{(\alpha_+,\alpha_-)}$ 
and $\CN^{(\alpha_+,\alpha_-,\beta_+,\beta_-,\lambda)}$ defined in~\eqref{eq:full_Mop}--\eqref{eq:full_Nop}
are compact as operators on  $X$,$\Cv(\surface)$ and on $\Hv^\half(\surface)$.
Furthermore,  for $\alpha$, $m$ and $s$ satisfying~\eqref{eq:conditions},
\begin{eqnarray*}
\CM^{(\alpha_+,\alpha_-)}, \CN^{(\alpha_+,\alpha_-,\beta_+,\beta_-,\lambda)}   &:&  \Cv^{m}(\surface) \rightarrow  \Cv^{m, \alpha}(\surface), \\
\CM^{(\alpha_+,\alpha_-)} , \CN^{(\alpha_+,\alpha_-,\beta_+,\beta_-,\lambda)} &:&  \Hv^{s-1}(\surface) \rightarrow  \Hv^{s}(\surface),
\end{eqnarray*}
are bounded linear operators. Consequently the $2 \times 2$ {\us{matrix-operator}} $\Mm$
defined in~\eqref{eq:Mm}
is a compact operator on $\Cv(\surface) \times \Cv(\surface)$ and on
$\Hv^\half(\surface) \times \Hv^\half(\surface)$. In addition,
if~\eqref{eq:conditions} holds then 
\begin{eqnarray}
\Mm  &:&  \Cv^{m}(\surface) \times  \Cv^{m}(\surface) \rightarrow  
\Cv^{m, \alpha}(\surface) \times \Cv^{m, \alpha}(\surface), \label{eq:Mm_c_prop} \\
\Mm &:&  \Hv^{s-1}(\surface) \times  \Hv^{s-1}(\surface) \rightarrow  
\Hv^{s}(\surface) \times \Hv^{s}(\surface), \label{eq:Mm_h_prop}
\end{eqnarray}
are bounded linear operators.
For $\x \in \surface$ define,
\begin{eqnarray}
\e^\i(\x)  &=& \frac{2\epsilon^+}{\epsilon^+ + \epsilon^-} \n(\x) \times \left(\gt^+ \E_\i\right)(\x) + 
\frac{2 \epsilon^-}{\epsilon^+ + \epsilon^-} \n(\x) \left(\gn^+ \E_\i\right)(\x),  \label{eq:einp}  \\
\h^\i(\x)  &=& \frac{2\mu^+}{\mu^+ + \mu^-} \n(\x) \times \left(\gt^+ \H_\i\right)(\x) + 
\frac{2 \mu^-}{\mu^+ + \mu^-} \n(\x) \left(\gn^+ \H_\i\right)(\x),  \label{eq:hinp} 
\end{eqnarray}
where $[\E_\i, \H_\i]$ is an incident electromagnetic wave.
In \cite{ganesh2014all} we proved that 
{\us{$(\f,\g)^T = (\e^\i, \h^\i)^T$}}
and  $(\e,\h)^T$ is  the exterior
trace of the total field 
solving the Maxwell dielectric problem
(\ref{eq:reduced-maxwell})--(\ref{eq:nor_interface}) for the same incident field $[\E_\i, \H_\i]$,
then 
\begin{equation} \label{eq:strawman}
\Big [ \Ima + \Mm  \Big ] \left (\begin{array}{l} \e \\ \h \end{array} \right ) =  \left (\begin{array}{l} \f \\ \g \end{array} \right ), 
\end{equation}
where $\Ima$ is the block identity operator.
This holds for all frequencies $\omega >0$.
{\it However, the system (\ref{eq:strawman}) 
may become singular for some frequencies $\omega$.}
In \cite{ganesh2014all} we found a constraint 
such that equation~\eqref{eq:strawman} augmented by that constraint is uniquely solvable 
for all frequencies. The stability constraint, which actually uses just
scalar operators that are also used in acoustics,
is governed by the operator $\Jm$ defined (with the same domain and range spaces
as $\Mm$) as: 
\begin{equation}\label{J-op}
\Jm \left (\begin{array}{l} \widetilde \e \\ \widetilde \h \end{array} \right ) =  \left  (0 ,~  0 ,~ \widetilde{\SK}\widetilde \e
       + \i \omega \mu^+  \widetilde{\SSS}  (\widetilde \h \cdot \n),~ 0  ,~ 0, ~  \widetilde{\SK} \widetilde \h
       - \i \omega \epsilon^+  \widetilde{\SSS}  (\widetilde \e \cdot \n)  \right )^T.
\end{equation}
For $\alpha$, $m$ and $s$ satisfying~\eqref{eq:conditions},
\begin{eqnarray}
\Jm  &:&  \Cv^{m}(\surface) \times  \Cv^{m}(\surface) \rightarrow  
\Cv^{m, \alpha}(\surface) \times \Cv^{m, \alpha}(\surface), \label{eq:Jm_c_prop} \\
\Jm &:&  \Hv^{s-1}(\surface) \times  \Hv^{s-1}(\surface) \rightarrow  
\Hv^{\textsc{}s}(\surface) \times \Hv^{s}(\surface), \label{eq:Jm_h_prop}
\end{eqnarray}
are bounded linear operators. Similar to $\Mm$, as $\omega \rightarrow 0$ the electric and magnetic fields
are decoupled in the range of the stabilization operator $\Jm$. 
We now  state the main result on the
well-posedness of the stabilized SIE model. 

\begin{theorem}
\label{theorem:well-posedness}
Let 
$[\f, \g]^T $ be the trace on $\surface$ 
of some incident field $[\E_\inc, \H_\inc] $.
{\revB{Assume that $\epsilon^\pm$ and $\mu^\pm$ are as 
specified in Section \ref{sec:int}.}}
Let $[\E, \H]$ be the unique solution of the associated penetrable scattering model 
\eqref{eq:reduced-maxwell}--\eqref{eq:EH_interface}.
Then  $[\e,\h]=[\gb^+ \E_\tot^+ ,\gb^+ \H_\tot^+]$ 
is the {\bf unique} solution    of the  weakly-singular SIE
\begin{equation}\label{eq:SIE_con}
\Big [ \Ima + \Mm  \Big ] \left (\begin{array}{l} \e \\ \h \end{array} \right ) =  \left (\begin{array}{l} \f \\ \g \end{array} \right ), 
\mbox{ subject to  } \quad \Jm \left (\begin{array}{l} \e \\ \h \end{array} \right ) =  \left (\begin{array}{l} \zero \\ \zero \end{array} \right ).
\end{equation}
In addition, if $\surface$ is of class $C^{m,1}$ with $m\geq 1$, then
$(\e,\h)^T \in \Cv^{m-1,1}(\surface) \times \Cv^{m-1,1}(\surface)$
 and depends continuously
 on
$(\f,\g)^T$.
Also $(\e,\h)^T \in \Hv^{m}(\surface) \times \Hv^{m}(\surface)$ and depends continuously
on $(\f,\g)^T$.
\end{theorem}

\begin{proof}
See \cite{ganesh2014all}.
\end{proof}
\vskip 10pt

{\section{A self-adjoint  SIE system without constraints{\us{: well-posedness and low
frequency stability}}}
\label{sec:J}

If the operator $\Ima + \Mm $  is singular,
then instead of solving equation~(\ref{eq:strawman}),
we could instead 
substitute $\Ima + \Mm + \xi \Jm$ for $\Ima + \Mm $, for a fixed complex number $\xi$ that 
would produce a regular operator 
{\revB{(that is, a continuous operator with continuous inverse)}}.
This is reasonable because
for right hand sides of interest,
derived from incident waves, we desire solutions
$[\e,\h]$ that satisfy the constraint
$\Jm \, (\e,\h)^T =  (\zero,\zero)^T$.
Unfortunately, however, we do not know a-priori
what value to choose for $\xi$. Once $\xi$ is fixed, if $\omega$ is changed
then the resulting operator
$\Ima + \Mm + \xi \Jm$ may become singular.
This was demonstrated numerically in 
\cite{ganesh2014all} for a spherical scatterer and,
in fact, we show in Appendix~\ref{Appendix1} that
this approach may fail for \textsl{any choice} 
of the complex scalar $\xi$.

More generally, solving linear problems with constraints has been an area of interest
for a long time and is surveyed in~\cite{benzi2005numerical}.
However, none of the cases covered in that survey paper applies to our particular case.
This is chiefly due to the fact that $\Mm$ is not symmetric, and that it is important for us not 
to lose the ``identity plus compact'' form of our operator.


Each of the operators introduced in Section~\ref{sec:oper}
has a formal adjoint, and
it is easy to verify that 
the adjoint operators $\Mm^*$ and $\Jm^*$ satisfy
the same continuity properties 
as $\Mm$ and $\Jm$.
To describe the continuity properties of $\Mm$, $\Jm$
  and their adjoints, it is
  helpful to introduce
the following notations:
the space $\Cv_2(\surface)$ will be defined to be the product space
$\Cv(\surface) \times\Cv(\surface)$.
$\Lv^2_2(\surface)$, $\Hv^s_2(\surface)$ and $\Cv^{m,\sigma}_2(\surface)$ are defined likewise.

{\revB{Just like $\Mm$ and $\Jm$, the operators $\Mm^*$, $\Jm^*$,
$\Mm^*\Mm$, and $\Jm^*\Jm$ can  be expressed
as  weakly singular integral operators; this is due 
to results in~\cite[Part I, Chapter 2, sections 1.1.2, 1.1.3]{maz1991boundary}.
It follows that these operators are also continuous and compact on
$\Cv^{m}(\surface) \times  \Cv^{m}(\surface)$ and on
$\Hv^{s}(\surface) \times  \Hv^{s}(\surface)$
provided that $\surface$ has at least regularity $C^{1, \sigma}$ with $\sigma >0$.}}

\vspace{0.2in}
We now state  our all-frequency stable  weakly-singular SIE system governed by an operator of the 
form $\Ima + \widetilde{\Mm}$, where  $\widetilde{\Mm}$ is a  self-adjoint operator
that incorporates the stabilization operator intrinsically.

\vspace{0.1in}
\begin{theorem}
\label{theorem:well-posedness2}
Let  $[\e,\h]$ 
be the {\bf unique} solution to equation (\ref{eq:SIE_con}).
Then $[\e,\h]$  is also the unique solution to the  equation
\begin{equation}\label{eq:SIE}
\Big [ \Ima + \Mm^* \Mm  + \Mm + \Mm^* + \Jm^* \Jm \Big  ] \left (\begin{array}{l} \e \\ \h \end{array} \right ) = \Big [ \Ima + \Mm^*  \Big  ]  \left (\begin{array}{l} \f \\ \g \end{array} \right ). 
\end{equation}
The solution 
$(\e,\h)^T$
depends continuously 
on $(\f,\g)^T$
and  this continuous
dependence is valid in any space
$\Cv^{m, \sigma}_2(\surface)$ or 
$\Hv^{s}_2(\surface)  $ with $m \in \mathbb{N}$, $\sigma \in [0,1]$, $s \geq 0$, 
provided that $\surface$ has sufficient
  regularity. 
\end{theorem}

\vskip 10 pt
\color{black}
The proof of Theorem \ref{theorem:well-posedness2}
first requires two general lemmas.
\begin{lemma}\label{well_posed}
Let $H$ be a Hilbert space and $\Mm,\Jm: H \to H$ two compact linear
    operators
such that $N(\Jm) \cap N(\Ima + \Mm) = \{ 0 \}$.
Then the linear operator $\Ima+ \Mm^* \Mm  + \Mm^* + \Mm + \Jm^* \Jm$ has a bounded inverse.
\end{lemma}
\begin{proof}
We note that $\Ima+\Mm^* \Mm +  \Mm^* + \Mm   = (\Ima+ \Mm)^*(\Ima+ \Mm)$.
If $x \in H$ is such that $(\Ima+ \Mm )^*(\Ima+ \Mm) x + \Jm^* \Jm x= 0$,
then $\| (\Ima+ \Mm) x\|^2 + \| \Jm x \|^2 =0$, so $x \in N(\Jm) \cap N(\Ima+ \Mm) $. It follows that, due to our assumption,
$x$ must be zero.
Since  $\Mm^* \Mm  + \Mm^* + \Mm + \Jm^* \Jm$ is compact, 
it follows that $\Ima+\Mm^* \Mm  + \Mm^* + \Mm + \Jm^* \Jm$
has a bounded inverse.
\end{proof}
\begin{lemma} 
Let $H, \Mm, \Jm$ be defined as in the previous lemma, with
$N(\Jm) \cap N(\Ima +\Mm) = \{ 0 \}$.
Assume also that the system of equations
\begin{equation} \label{eq}
\left\{
\begin{array}{l}
(\Ima + \Mm) x= b \\
\Jm x=0
\end{array}
\right.
\end{equation}
has a solution $x_0 \in H$.
Then 
$x_0$ is the unique solution to the equation
\begin{equation}  \label{eq2}
(\Ima+ \Mm^* \Mm  + \Mm^* + \Mm + \Jm^* \Jm)x = (\Ima+ \Mm^*) b.
\end{equation}
\end{lemma}
\begin{proof}
Uniqueness follows from Lemma \ref{well_posed}.
Since $x_0$ satisfies (\ref{eq}) it is clear that 
$(\Ima+ \Mm^* )(\Ima+ \Mm) x_0 = (\Ima+ \Mm^* ) b$ and $\Jm^* \Jm x_0 =0$.
Thus
$x_0$ satisfies (\ref{eq2}).
\end{proof}
Theorem \ref{theorem:well-posedness2} is now proved
thanks to the previous lemma and Theorem \ref{theorem:well-posedness}.
There remains an important question: do we have a control over the
norm of the inverse operator  $(\Ima+ \Mm^* \Mm  + \Mm^* + \Mm + \Jm^* \Jm)^{-1}$?
To address this question, we now show that
the norm of this inverse operator
can be estimated using the norms of $\Mm$ and $\Jm$, a constant that measures 
"how well" $(\Ima+ \Mm)$ is invertible on the orthogonal of $N(\Ima+ \Mm)$,
and  another constant that measures ``how 
positive'' $\Jm^* \Jm$ is on $N(\Ima+ \Mm)$.

We first note that $\Ima+ \Mm$ is bijective from $N(\Ima+ \Mm)^\perp$ to $R(\Ima+ \Mm)$, 
and that its inverse is bounded. Denote by $P$ the orthogonal projection on $N(\Ima+ \Mm)$.
There is a positive constant $C_1$ such that for all $x \in H$,
\bea \label{C1def}
\| (\Ima - P) x\| \leq C_1  \| (\Ima + \Mm) (\Ima - P) x\|.
\eea
Next we note that $N(\Ima+\Mm)$ is finite  dimensional.
Using that $N(\Jm) \cap N(\Ima +\Mm) = \{ 0 \}$, 
$\| \Jm  \|$ achieves a strictly positive minimum
on the compact set $\{ x: x \in N(\Ima+\Mm), \, \| x\| =1\}$.
Therefore, there is a positive constant $C_2$ such that for all $x \in H$,
\bea \label{C2def}
\|  P x\| \leq C_2  \| \Jm P x\|.
\eea
Assume now that
\bea
(\Ima+\Mm)^*(\Ima+\Mm) x + \Jm^*\Jm x = d, \label{xd}
\eea
for some $d$,
and estimate $x$.\\

There are two cases.
Let us first consider the case when $d \in N(\Ima+\Mm)^\perp$, 
that is,
$d = (\Ima-P) d$.
Evaluating the dot product of $x$ with (\ref{xd}), we obtain 
$$
\|(\Ima+\Mm) x\|^2 + \|\Jm x \|^2= \langle x,(\Ima-P)d \rangle,
$$
which we re-write as
$$
\|(\Ima+\Mm) (\Ima-P)x\|^2 + \|\Jm x \|^2= \langle (\Ima-P)x,d \rangle.
$$
Using (\ref{C1def}),
$$
\frac{1}{C_1^{2}} \| (\Ima-P)x\|^2 + \|\Jm x \|^2 \leq \| (\Ima-P)x\| \| d\|.
$$
Thus, we obtain 
\bea \label{I-Px}
\| (\Ima-P)x\| \leq C_1^{2} \| d\|, \qquad   \|\Jm x \|\leq C_1 \| d\|.
\eea
It now follows that,
$$
\| \Jm P x \| \leq \|\Jm x \| + \| \Jm (\Ima-P)x\| \leq \left( C_1 +
 C_1^2 \| \Jm \| \right)\| d\|,
$$
and using (\ref{C2def}),
\bea \label{Px}
\|P x\| \leq C_2 \left( C_1 +
 C_1^2 \| \Jm \| \right)\| d\|.
\eea
Thus, we obtain the bound
\bea
\|x\|^2 = \|P x\| ^2 + \|(\Ima-P) x\|^2 \leq 
\left(C_2^2 \left( C_1 +
 C_1^2 \| \Jm \|\right)^2 + C_1^4\right)\| d\|^2.
\label{bound1}
 \eea

Let us now consider the case when $d \in N(\Ima+\Mm)$ and estimate $\| x\|$
from (\ref{xd}).
Using that $Pd=d$ and evaluating the inner product of 
equation (\ref{xd}) with $x$, we obtain,
\bea \label{dp1}
\|(\Ima+\Mm) (\Ima-P)x\|^2 + \|\Jm x \|^2= \langle Px,d \rangle.
\eea
Next, we evaluate the inner product of 
equation (\ref{xd}) with $Px$ to obtain
\bea \label{dp2}
\langle \Jm Px , \Jm x \rangle= \langle Px,d \rangle.
\eea
Combining (\ref{dp1}) with (\ref{dp2}) we arrive at
\bea \label{dp3}
\|(\Ima+\Mm) (\Ima-P)x\|^2 + \frac12 \|\Jm x \|^2 \leq \frac12 \|\Jm Px \|^2.
\eea
From (\ref{dp1}) and (\ref{C1def})
\bea \label{dp4}
\|(\Ima-P)x\|^2 \leq C_1^2 \langle Px, d \rangle .
\eea
Next, we note that, using (\ref{dp2}),
\bea \label{dp5}
\|\Jm Px \|^2 = \langle Px,d \rangle - \langle \Jm Px, \Jm (\Ima-P) x\rangle \no 
\leq \langle Px,d\rangle  + \frac12 \| \Jm Px \|^2 + \frac12 \|\Jm (\Ima-P) x \|^2,
\eea
so combining (\ref{dp4}) and (\ref{dp5}), 
\bea
\frac12 \|\Jm Px \|^2 \leq \left( 1 + \frac12 C_1^2 \|\Jm \|^2 \right) \langle Px,d\rangle ,
\eea
and recalling (\ref{C2def}), 
\bea \label{dp7}
\|Px \|^2 \leq  C_2^2 \left(2 +  C_1^2 \|\Jm \|^2 \right) \langle Px,d\rangle \no 
\leq \frac12  \| Px \|^2
+  \frac{1}{2 } C_2^4 (2 +  C_1^2 \|\Jm \|^2)^2 \| d \|^2,
\eea
thus
\bea \label{dp8}
\|Px \|^2 \leq C_2^4 (2 +  C_1^2 \|\Jm \|^2)^2 \| d \|^2.
\eea
Using (\ref{dp4}) one more time, it follows that
\bea \label{dp9}
\|(\Ima-P)x \|^2 \leq C_1^2 C_2^2 (2 + C_1^2 \| \Jm \|^2) \|d\|^2.
\eea
Combining (\ref{dp8}) and  (\ref{dp9}), it is clear that $\| x \|$
is bounded by a constant depending only on $C_1, C_2,$ and
$\| \Jm \| \|d\|$.

We covered  two cases: if  $d \in N(\Ima + \Mm)^\perp$ then  
the estimate~(\ref{bound1}) holds, and if  $d \in N(\Ima + \Mm)$ 
then the estimate~(\ref{dp8}-\ref{dp9}) holds.
 Since $H$ is the direct sum of
$N(\Ima + \Mm)^\perp$ and $N(\Ima + \Mm)$ and equation 
(\ref{xd}) is linear, the estimate for $d \in H$ results from combining the previous
two estimates.

\begin{remark}
	In the trivial case where the nullspace of \texorpdfstring{$\Ima+\Mm$}{the second kind operator} is reduced to zero,
we have that
$(\Ima-P) = \Ima$.  Consequently, (\ref{I-Px})  reduces to 
$\| x\| \leq  C_1^{2} \| d\|$, that is, the norm of 
$(\Ima+ \Mm^* \Mm  + \Mm^* + \Mm + \Jm^* \Jm)^{-1}$ 
is bounded by $C_1^{2}$. 
\end{remark}

\vskip 10pt
\color{black}
Although it is clear that the operator 
$\Ima + \Mm^* \Mm  + \Mm^* + \Mm + \Jm^* \Jm$ 
is bounded as $\omega$ tends to zero, we have yet to argue that its inverse
 too remains bounded in this limiting case.
As mentioned in the introduction, this will prove that our
integral equation formulation does not suffer  from the well documented
low frequency breakdown. We emphasize that this property holds for the continuous 
operators and thus any numerical method that correctly approximates these operators
will inherit this no breakdown property.
In this paragraph we temporarily mark the dependence of operators on the frequency 
$\omega$ explicitly.

\begin{theorem}
\label{no low freq break down}
For $\omega >0$ the operator $\Ima + \Mm^*(\omega) \Mm(\omega)  +
 \Mm^*(\omega) + \Mm(\omega) + \Jm^*(\omega) \Jm(\omega) $ 
acting on $\Cv^{m , \sigma}_2(\surface)$ or on
$\Hv^{s}_2(\surface)  $ depends continuously on $\omega$. 
As $\omega$ tends to zero,
it converges strongly to the operator 
$\Ima + \Mm^*(0) \Mm(0)  +
 \Mm(0) + \Mm^*(0) $.
Assuming that $\epsilon^\pm$ and $\mu^\pm$ are as 
specified in Section \ref{sec:int}, this latter operator
 is invertible and its inverse is continuous. 
Accordingly, the inverse of $\Ima + \Mm^*(\omega) \Mm(\omega)  +
 \Mm^*(\omega) + \Mm(\omega) + \Jm^*(\omega) \Jm(\omega) $ 
remains uniformly bounded as $\omega \rightarrow 0$.
\end{theorem}
\begin{proof}
Denote
\bea
G_0(\x,\y)  =\frac{ 1}{4 \pi|\x - \y|}.
\eea
Write

\begin{eqnarray}
 \epsilon^+\pm G_+(x,y) - \epsilon^-\pm G_-(x,y) - 
(\epsilon^+ - \epsilon^-)G_0(x,y) 
= & \omega R(x,y)
\end{eqnarray}
where the remainder $R$ is given by
\begin{eqnarray}
R(x,y)=\sum_{l=1}^\infty
\frac{1}{l!} \omega^{l-1} |x-y|^{l-1} i^l((\mu^+\epsilon^+)^{\frac{l}{2}}-
(\mu^-\epsilon^-)^{\frac{l}{2}}).
\end{eqnarray}
Clearly, $R$ is uniformly bounded on $\surface \times \surface $ for $\omega$
in $(0,1)$. For $x\neq y$ in $\surface $, $R(x,y)$
has infinitely many derivatives and there is an upper bound 
for the supremum of these derivatives that is independent of $\omega$. 
It follows that the  boundary  integral operator associated to $R$
is continuous on $\Cv^{m , \sigma}_2(\surface)$ or on
$\Hv^{s}_2(\surface)  $ and its norm is bounded above by a constant independent of $\omega$ in $(0,1)$.
Recall the definition of the operator $\SD_\pm$
introduced in~\eqref{eq:D}  which is involved in~\eqref{eq:C} and hence
further in the definition of $\widetilde{\SC}^{(\alpha^+,\alpha^-)}$
in~\eqref{eq:sub_operator}. 
The argument above shows that
the operator $\epsilon_+ \SD_+ - \epsilon_- \SD_-$ is norm 
convergent as $\omega \rightarrow 0 $ to the operator 
\bea
\w \rightarrow (\epsilon_+ - \epsilon_- )\int_{\surface} {\frac{ \partial G_0}{\partial \nx} }(\x,\y)  \; \w(\y) \; ds(\y).
\eea
This argument can be carried over to all the other operators involved in the
definition of $\Mm(\omega)$ to show that it converges in operator norm
to $\Mm(0)$ as $\omega \rightarrow 0$
 and that $\Jm(\omega)$ converges in operator norm
to zero. In ~\cite[Appendix B]{ganesh2014all} we gave an explicit
formulation for $\Mm(0)$ and we proved that 
 $\Ima + \Mm(0) $ has a bounded inverse.
Altogether, we can claim that 
$\Ima + \Mm^*(\omega) \Mm(\omega)  +
 \Mm^*(\omega) + \Mm(\omega) + \Jm^*(\omega) \Jm(\omega) $ 
is norm convergent to 
\bea
\Ima + \Mm^*(0) \Mm(0)  +
 \Mm^*(0) + \Mm(0) = (\Ima + \Mm(0) )^*(\Ima + \Mm(0)) .
\eea
Since this operator has a bounded inverse, the proof is complete.
\end{proof}
\color{black}

{\revB{
\begin{remark}
In ~\cite[Appendix B]{ganesh2014all} no assumption was necessary 
on the genus of $\surface$. The uniform well-posedness
statement of Theorem \ref{no low freq break down}
therefore holds regardless of the genus  of $\surface$. 
\end{remark}
}}

\color{black}
\section{Numerical algorithm and convergence analysis}
\label{sec:spectral}

For our numerical approach,
we transform all of our integral operators to a reference surface $S$,
the unit sphere.
In particular, we use a bijective parametrization $\q:S \to \partial D$,
which is
not necessarily the standard spherical polar coordinates transformation,
to transform the operators $\Mm$ and $\Jm$ to
operators acting on spherical functions.
To simplify  notations, the resulting operators will still be denoted by 
$\Jm$ and $\Mm$. 
Some smoothness is required of the transformation $\q$, which we describe
in the final convergence theorem in this section.
Note, however, that in atmospheric physics~\cite{mish94,kahnert,rother:cheby},
the shapes of relevant particles satisfy this smoothness requirement.


\subsection{Vector spherical polynomials}

Our approximation to the surface fields on the reference surface
using a finite dimensional space is based on
the standard
spherical harmonic basis $Y_{l,j}$,
which is orthogonal
with respect to
the natural $L^2$ inner product $(\cdot,\cdot)$ on the unit sphere $S$.
In particular, to approximate the fields $\e$ and $\h$ in~\eqref{eq:SIE},
we define
$\Y_{l,j,k} = Y_{l,j} \e_k$, where $\e_k$ for $k=1,\dots, 3$ is the natural basis
of $\mathbb{C}^3$. 
Let $\underline{\mathbb{P}}_n$ be the span of $\Y_{l,j,k} $ for
$ 0\leq l \leq n, \ |j| \leq l, \ 1 \leq k \leq 3$.
As emphasized in~\cite{gh:first},
for discrete approximation of the inner product on $S$ within
the algorithm,
it is essential to use a quadrature rule with $m$ points
$\hat{\x}_q^m$ and weights $\zeta_q^m$ for $q=1 , .., m$  that
computes inner products between
any two elements in $\underline{\mathbb{P}}_n$ exactly.
Such a quadrature rule with $m=2(n+1)^2$
can be found in \cite{gh:first}.

In particular, if we denote by $(\cdot,\cdot)_m$ the bilinear product obtained by approximating
the inner product $(\cdot,\cdot)$ on $\Cv_2(S)$ using the 
points $\hat{\x}_q^m$ and weights $\zeta_q^m$ for $q=1 , .., m$, then
\bea
(\G, \H)_m = (\G, \H), \quad \forall \, \G, \H \in \underline{\mathbb{P}}_n.
\label{exact}
\eea
In line with~\cite[Page~34]{gh:first},
we will denote by $\underline{{\cal L}}_n:\Cv_2(S) \to \underline{\mathbb{P}}_n$
the vector-valued fully-discrete orthogonal projection
obtained 
using the discrete inner product $(\cdot,\cdot)_m$. 

The key ingredient in our algorithm is to split
$\Mm$ into operators with weakly singular kernels
$\Mm_1$ and
smooth kernels $\Mm_2$ respectively.
Details of such a splitting are given in~\cite[Section~3]{dielectric2}.
Following details in~\cite[Section~4]{dielectric2},
for integer $n' \geq n+2$ we approximate the singular part $\Mm_1$
and regular part $\Mm_2$ of $\Mm$ by
$\Mm_{1,n'}$ and $\Mm_{2,n'}$ respectively.
We similarly approximate $\Jm$ by $\Jm_{n'}$ (and note that 
$\Jm$ is regular and so has no singular part).

Our fully discrete numerical algorithm to compute the parametrized approximation
$\Phi_n = (\e_n \circ \q,\h_n \circ \q)^T$ to the  unique solution
$\Phi = (\e \circ q,\h \circ \q)^T$, defined on $\sphere \times \sphere$, 
of the transformed single-equation~\eqref{eq:SIE}  on $\sphere$ with
right hand side $F = [\Ima + \Mm^*] (\f \circ \q,\g \circ \q)^T$ can be expressed using the
orthogonal projection operator as
\begin{equation}
  \label{eq:main}
(\Ima + \underline{{\cal L}}_n \Mm^*_{n'} \underline{{\cal L}}_n \Mm_{n'}  \underline{{\cal L}}_n +
\underline{{\cal L}}_n \Mm^*_{n'} \underline{{\cal L}}_n +
\underline{{\cal L}}_n \Mm_{n'}  \underline{{\cal L}}_n
+ \underline{{\cal L}}_n \Jm^*_{n'} \underline{{\cal L}}_n \Jm_{n'}  \underline{{\cal L}}_n)  \Phi_n = \underline{{\cal L}}_n F. 
\end{equation}
We omit a detailed description of how to implement the algorithm, which
follows closely the detailed descriptions in~\cite{gh:first,dielectric2}.
In the rest of this section we focus on proving spectrally accurate
convergence of $\Phi_n$ to $\Phi$ as $n \to \infty$ under certain
smoothness assumptions on $F$, that are dictated by the smoothness of the
surface.


To prove well-posedness of the fully-discrete system~\eqref{eq:main}, 
we first establish convergence of the approximate surface integral operators
in an appropriate norm.
Denote by $\| \cdot \|$ the supremum norm on $\Cv_2(S)$.
The corresponding induced norm 
on the space of continuous linear operators
from $\Cv_2(S)$ to $\Cv_2(S)$
will also be denoted by $\| \cdot \|$.
The following two estimates for the vector-valued fully discrete projection operator were derived in~\cite[Appendix~A]{gh:first}, using
associated scalar counterpart results in~\cite{sloan:constructive}:
\bea \label{estGan1}
\|\psi - \underline{{\cal L}}_n \psi \| \leq C n^{1/2} \|\psi\|,
\eea 
for $\psi \in C(S)$, where
$C$ is independent of $n$, and
\bea \label{estGan2}
\|\psi - \underline{{\cal L}}_n \psi \| \leq C n^{1/2  - r - \alpha} \|\psi\|_{\Cv_2^{r,\alpha}},
\eea 
for $\psi \in \Cv_2^{r,\alpha}(S)$,
where $C$ is independent of $n$.

\begin{theorem}
For any $p$ in $\mathbb{N}$, there is a positive constant $c_p$ independent of $n$
and $n' \geq n-2$ such that
\bea
\| (\Mm - \Mm_{1,n'} - \Mm_{2,n'}) \P_n\| + \| (\Jm - \Jm_{n'}) \P_n\|
\leq c_p n^{-p} \| \P_n \|, \label{spec1}
\eea
for all $\P_n$ in $\underline{\mathbb{P}}_n$.
\end{theorem}

\begin{proof}
  See~\cite{gh:first}.
\end{proof}

To simplify the notation
set $\Mm_{n'}= \Mm_{1,n'} + \Mm_{2,n'}$. Using (\ref{estGan1}), 
there is an interesting way to rewrite  estimate (\ref{spec1}) in terms of norm
operators from $\Cv_2(S)$ to $\Cv_2(S)$,
\bea
\| (\Mm - \Mm_{n'} ) \underline{{\cal L}}_n\| + \| (\Jm - \Jm_{n'}) \underline{{\cal L}}_n\|
\leq c_p n^{1/2-p}. \label{spec2}
\eea
The adjoint operators $\Mm^*$ and $\Jm^*$ can also be approximated as in \cite{gh:first} 
by $\Mm^*_{n'}$ and $\Jm^*_{n'}$, and the following estimate holds
\bea
\| (\Mm^* - \Mm^*_{n'} ) \underline{{\cal L}}_n\| + \| (\Jm^* - \Jm^*_{n'}) \underline{{\cal L}}_n\|
\leq c_p n^{1/2-p}. \label{specstar}
\eea

Proving existence and uniqueness of the solution of the discrete
  analogue of
  $ (\Ima + \Mm^* \Mm  + \Mm + \Mm^* + \Jm^* \Jm ) \Phi = F$
  in the space $\underline{\mathbb{P}}_n$,
  and that the solution converges to $\Phi$,
  is rather involved.
The main hurdle is that the norm of the 
projection operator $\underline{{\cal L}}_n$
from  $\Cv_2(S)$ to $\underline{\mathbb{P}}_n$ becomes unbounded as $n$ grows large.
 In fact, it was shown in~\cite{sloan:constructive} that this norm is exactly of the order $n^{1/2}$ (for a scalar counterpart of $\underline{{\cal L}}_n$).

\begin{lemma} \label{step1}
  For all sufficiently large $n$, 
  the linear operator
  \begin{displaymath}
 \Ima +   \Mm^*\underline{{\cal L}}_n \Mm \underline{{\cal L}}_n + 
\Mm^* \underline{{\cal L}}_n + \Mm \underline{{\cal L}}_n + 
\Jm^*\underline{{\cal L}}_n \Jm\underline{{\cal L}}_n:
\Cv_2(S) \to \Cv_2(S)
    \end{displaymath}
is invertible.
For all sufficiently large $n$,
  \begin{align*}
 \| (\Ima +   \Mm^*\underline{{\cal L}}_n \Mm \underline{{\cal L}}_n + 
\Mm^* \underline{{\cal L}}_n + \Mm \underline{{\cal L}}_n +
\Jm^*\underline{{\cal L}}_n \Jm\underline{{\cal L}}_n )^{-1} \|
& \leq C n^{1/2}, \\
\| \Mm^*\underline{{\cal L}}_n \Mm \underline{{\cal L}}_n + 
\Mm^* \underline{{\cal L}}_n + \Mm \underline{{\cal L}}_n +
\Jm^*\underline{{\cal L}}_n \Jm\underline{{\cal L}}_n \|
& \leq C n^{1/2},
  \end{align*}
where $C$ is a constant.
\end{lemma}

\begin{proof}
Because $\Mm,\Jm:\Cv_2(S) \to \Cv_2^\alpha(S)$ are continuous operators for 
any $\alpha$ in $(0,1)$ 
so in particular for $\alpha$ in $(\frac{1}{2},1)$), 
it follows from (\ref{estGan2})
that
\bean
\| (\Ima - \underline{{\cal L}}_n) \Mm \|+ \|(\Ima - \underline{{\cal L}}_n) \Jm \|
+ \| (\Ima - \underline{{\cal L}}_n) \Mm^* \|
 \rightarrow 0,
 \quad \mbox{as $n \to \infty$,}
\eean
so 
\bean
\| \Mm^*(\Ima - \underline{{\cal L}}_n) \Mm 
+ (\Ima - \underline{{\cal L}}_n) \Mm^* + (\Ima - \underline{{\cal L}}_n) \Mm +
\Jm^*(\Ima - \underline{{\cal L}}_n) \Jm \| \rightarrow 0,
\quad \mbox{as $n \to \infty$.}
\eean
Then using that $\Ima + \Mm^* \Mm  + \Mm +\Mm^* + \Jm^* \Jm$ is invertible, it follows
that $ \Ima +  \Mm^*\underline{{\cal L}}_n \Mm + \underline{{\cal L}}_n \Mm^* + 
\underline{{\cal L}}_n \Mm+ \Jm^*\underline{{\cal L}}_n \Jm$
is invertible with $n > N$ for some $N$, and 
$ \|(\Ima +  \Mm^*\underline{{\cal L}}_n \Mm + \underline{{\cal L}}_n \Mm^* + 
\underline{{\cal L}}_n \Mm+ \Jm^*\underline{{\cal L}}_n \Jm)^{-1} \|$
is uniformly bounded for $n>N$ and convergent to $ \|(\Ima +  \Mm^* \Mm + \Mm^* +
\Mm + \Jm^* \Jm)^{-1} \|$.
Let $\Phi$ be in $\Cv_2(S)$ and fix 
$\alpha \in (1/2,1)$. Then using (\ref{estGan2}),
\bean
\| (\Ima - \underline{{\cal L}}_n) \Mm^* \underline{{\cal L}}_n \Mm \Phi \| & \leq &
C n^{1/2 - \alpha} \|  \Mm^* \underline{{\cal L}}_n \Mm \Phi \|_{\Cv_2^\alpha} \\
 & \leq & C n^{1/2 - \alpha} \|   \underline{{\cal L}}_n \Mm \Phi \| \\
& \leq &  C n^{1/2 - \alpha} \|   (\Ima -\underline{{\cal L}}_n ) \Mm \Phi \|
 + C n^{1/2 - \alpha} \|    \Mm \Phi \|  \\
& \leq &  C n^{1-  2\alpha} \|    \Mm \Phi \|_{C^\alpha}
 + C n^{1/2 - \alpha} \|    \Mm \Phi \|  \\
& \leq &  C n^{1-  2\alpha} \|   \Phi \|
 + C n^{1/2 - \alpha} \|     \Phi \|.
\eean
Thus $\|(\Ima - \underline{{\cal L}}_n)\Mm^*\underline{{\cal L}}_n \Mm \| \to 0$
as $n \to \infty$. 
A similar result holds for $\|(\Ima - \underline{{\cal L}}_n)\Jm^*\underline{{\cal L}}_n \Jm \|$.
It follows
that $ \Ima + \underline{{\cal L}}_n \Mm^*\underline{{\cal L}}_n \Mm + 
\underline{{\cal L}}_n \Mm^* + \underline{{\cal L}}_n \Mm+
\underline{{\cal L}}_n \Jm^*\underline{{\cal L}}_n \Jm$
is invertible for $n \geq N_2$ for some $N_2$, and
$ \|(\Ima + \underline{{\cal L}}_n \Mm^*\underline{{\cal L}}_n \Mm + 
\underline{{\cal L}}_n \Mm^* + \underline{{\cal L}}_n \Mm+
\underline{{\cal L}}_n \Jm ^*\underline{{\cal L}}_n \Jm)^{-1} \|$
is uniformly bounded for $n>N_2$ and convergent to 
$ \|(\Ima +  \Mm^* \Mm + \Mm^* + \Mm + \Jm^* \Jm)^{-1} \|$.\\

A simple calculation using the identity
\bean
  (I +BA) (I -B(I + AB )^{-1}A)  =(I -B(I + AB )^{-1}A) (I +BA) =I
	\eean
shows that
the inverse of $\Ima +   \Mm^*\underline{{\cal L}}_n \Mm \underline{{\cal L}}_n+ 
 \Mm^*\underline{{\cal L}}_n + \Mm \underline{{\cal L}}_n+ 
 \Jm^*\underline{{\cal L}}_n \Jm \underline{{\cal L}}_n $
 is
 $$
\Ima - (\Mm^*\underline{{\cal L}}_n \Mm + \Mm^* + \Mm+  \Jm^*\underline{{\cal L}}_n \Jm)
 (\Ima + \underline{{\cal L}}_n \Mm^*\underline{{\cal L}}_n \Mm + 
\underline{{\cal L}}_n \Mm^* + \underline{{\cal L}}_n \Mm+
\underline{{\cal L}}_n\Jm^*\underline{{\cal L}}_n \Jm)^{-1} \underline{{\cal L}}_n.
$$
At this stage we recall that 
$ \|(\Ima + \underline{{\cal L}}_n \Mm^*\underline{{\cal L}}_n \Mm +
 \underline{{\cal L}}_n \Mm^* + \underline{{\cal L}}_n \Mm +
 \underline{{\cal L}}_n\Jm^*\underline{{\cal L}}_n \Jm)^{-1} \|$ and 
$\|\Mm^*\underline{{\cal L}}_n \Mm + \Mm^* + 
\Mm + \Jm^*\underline{{\cal L}}_n \Jm\|$ are bounded, and
that $\|\underline{{\cal L}}_n \| \leq C n^{1/2}$ to conclude that
\bea
\| (\Ima +   \Mm^*\underline{{\cal L}}_n \Mm \underline{{\cal L}}_n+ 
 \Mm^*\underline{{\cal L}}_n + \Mm \underline{{\cal L}}_n+ 
\Jm^*\underline{{\cal L}}_n \Jm \underline{{\cal L}}_n)^{-1}\| \leq Cn^{1/2}. \label{Nln}
\eea
Similar arguments  show that 
 $\| \Mm^*\underline{{\cal L}}_n \Mm \underline{{\cal L}}_n + 
\Mm^* \underline{{\cal L}}_n + \Mm \underline{{\cal L}}_n +
 \Jm^*\underline{{\cal L}}_n \Jm\underline{{\cal L}}_n \|
\leq Cn^{1/2}$.
\end{proof}

\begin{theorem}
The linear operator 
\bean
 \Ima +  \underline{{\cal L}}_n \Mm^*_{n'}   \underline{{\cal L}}_n \Mm_{n'}  \underline{{\cal L}}_n
+  \underline{{\cal L}}_n \Mm^*_{n'}   \underline{{\cal L}}_n +
\underline{{\cal L}}_n \Mm_{n'}  \underline{{\cal L}}_n 
 + \underline{{\cal L}}_n \Jm^*_{n'}   \underline{{\cal L}}_n \Jm_{n'}  \underline{{\cal L}}_n: \underline{\mathbb{P}}_n \to \underline{\mathbb{P}}_n
\eean
is invertible for all sufficiently large $n$  and $n' \geq n+2$.
Additionally, if $\Phi \in C(S)$ is the unique solution to
\bea
 (\Ima + \Mm^*\Mm + \Mm^* + \Mm + \Jm^*\Jm) \Phi = F, \label{Phieq}
\eea
where, for $r > 2$ (depending on the smoothness of the 
parametrization map $\q$),  $F$ is in $C^{r, \alpha}(S)$ 
with $\alpha \in (1/2,1)$,  and if for sufficiently large $n$,
$\Phi_n \in \underline{\mathbb{P}}_n$
solves 
\bea \label{discrete_eq}
( \Ima +  \underline{{\cal L}}_n \Mm^*_{n'}   \underline{{\cal L}}_n \Mm_{n'}  \underline{{\cal L}}_n
+  \underline{{\cal L}}_n \Mm^*_{n'}   \underline{{\cal L}}_n +
\underline{{\cal L}}_n \Mm_{n'}  \underline{{\cal L}}_n 
 + \underline{{\cal L}}_n \Jm^*_{n'}   \underline{{\cal L}}_n \Jm_{n'}  \underline{{\cal L}}_n)
\Phi_n = \underline{{\cal L}}_n F, \label{discrete}
\eea
then there is a constant $C$ independent of $n$ and $n' \geq n +2$ 
such that 
\bea
\| \Phi - \Phi_n \| \leq
C n^{5/2-r - \alpha}  \| F \|_{\Cv_2^{r,\alpha}}. \label{main_est}
\eea

\end{theorem}

\begin{proof}
From (\ref{estGan1}), we note that $\underline{{\cal L}}_n: \Cv_2(S) \to \Cv_2(S)$ is a bounded linear operator
and 
\bea
\|\underline{{\cal L}}_n \| \leq C n^{1/2}. \label{sqrt}
\eea
Combining this with~(\ref{spec2}), establishes
\bean
\| \underline{{\cal L}}_n(\Mm - \Mm_{n'} ) \underline{{\cal L}}_n\| \leq C n^{1-p},
\eean
where $C$ depends on $p$.
Then
\bea
\| \Mm^*\underline{{\cal L}}_n(\Mm - \Mm_{n'} ) \underline{{\cal L}}_n\| \leq C n^{1-p},
\label{starcombo1}
\eea
and, combining (\ref{spec2}) with (\ref{specstar}), 
\bea
\| (\Mm^* - \Mm^*_{n'} )\underline{{\cal L}}_n(\Mm - \Mm_{n'} ) 
\underline{{\cal L}}_n\| \leq C n^{1-p}.
\label{starcombo2}
\eea
Combining (\ref{starcombo1}) and  (\ref{starcombo2}) yields
\bea
\| (\Mm^* + \Mm^*_{n'} )\underline{{\cal L}}_n(\Mm - \Mm_{n'} ) 
\underline{{\cal L}}_n\| \leq C n^{1-p}.
\label{plus}
\eea
Next, we write
\bean
\lefteqn{\Mm^*_{n'} \underline{{\cal L}}_n \Mm \underline{{\cal L}}_n - \Mm^*  \underline{{\cal L}}_n 
\Mm_{n'} \underline{{\cal L}}_n } & & \\
& = &  (\Mm^*_{n'} - \Mm^*) \underline{{\cal L}}_n \Mm \underline{{\cal L}}_n 
- \Mm^* \underline{{\cal L}}_n (\Mm_{n'} - \Mm) \underline{{\cal L}}_n,
\eean
so, by another application of~(\ref{spec2})--(\ref{specstar}) and (\ref{sqrt}) we have
\bea
\| \Mm^*_{n'} \underline{{\cal L}}_n \Mm \underline{{\cal L}}_n - \Mm^*  \underline{{\cal L}}_n 
\Mm_{n'} \underline{{\cal L}}_n\| \leq C n^{1-p}. \label{mixed}
\eea
We now combine (\ref{plus}) with (\ref{mixed}) to obtain
\bea
\|  \Mm^* \underline{{\cal L}}_n \Mm \underline{{\cal L}}_n - \Mm^*_{n'} \underline{{\cal L}}_n 
\Mm_{n'}  \underline{{\cal L}}_n\| \leq C n^{1-p}. \label{almost}
\eea
Similarly,
\bea
\|  \Jm^* \underline{{\cal L}}_n \Jm \underline{{\cal L}}_n - \Jm^*_{n'} \underline{{\cal L}}_n 
\Jm_{n'}  \underline{{\cal L}}_n\| \leq C n^{1-p}.  \label{almost3}
\eea
Finally,
choosing $p>3/2$ and recalling \eqref{spec2}--\eqref{specstar}
we conclude using Lemma~\ref{step1} that
$\Ima +   \Mm^*_{n'}   \underline{{\cal L}}_n \Mm_{n'}  \underline{{\cal L}}_n
+  \Mm^*_{n'}   \underline{{\cal L}}_n + \Mm_{n'}  \underline{{\cal L}}_n 
 + \Jm^*_{n'}   \underline{{\cal L}}_n \Jm_{n'}  \underline{{\cal L}}_n$
is invertible for sufficiently large $n$  and for those $n$ thereupon,
\bea \label{inv_est}
\| (\Ima +   \Mm^*_{n'}   \underline{{\cal L}}_n \Mm_{n'}  \underline{{\cal L}}_n
+ \Mm^*_{n'}   \underline{{\cal L}}_n  + \Mm_{n'}  \underline{{\cal L}}_n 
 + \Jm^*_{n'}   \underline{{\cal L}}_n \Jm_{n'}  \underline{{\cal L}}_n)^{-1}
\| \leq C n^{1/2}.
\eea
Now let $F \in \Cv_2(S)$, and $\Psi_n$ solve
\bea
(\Ima + \Mm^*_{n'} \underline{{\cal L}}_n \Mm_{n'}  \underline{{\cal L}}_n +
\Mm^*_{n'} \underline{{\cal L}}_n + \Mm_{n'}  \underline{{\cal L}}_n
+ \Jm^*_{n'} \underline{{\cal L}}_n \Jm_{n'}  \underline{{\cal L}}_n)  \Psi_n = \underline{{\cal L}}_n F.
\label{eqPhin}
\eea
As $(\underline{{\cal L}}_n)^2 =\underline{{\cal L}}_n $,
left-multiplying equation~(\ref{eqPhin}) by $\underline{{\cal L}}_n$
shows
that $\underline{{\cal L}}_n \Psi_n$ is the solution to 
\bea
(\Ima + \underline{{\cal L}}_n \Mm^*_{n'} \underline{{\cal L}}_n \Mm_{n'}  \underline{{\cal L}}_n +
\underline{{\cal L}}_n \Mm^*_{n'} \underline{{\cal L}}_n +
\underline{{\cal L}}_n \Mm_{n'}  \underline{{\cal L}}_n
+ \underline{{\cal L}}_n \Jm^*_{n'} \underline{{\cal L}}_n \Jm_{n'}  \underline{{\cal L}}_n)  \Phi_n = \underline{{\cal L}}_n F. \label{eqPhin2}
\eea
We conclude that equation (\ref{discrete_eq}) is always solvable,
and the solution is unique because
$\underline{\mathbb{P}}_n$ is finite dimensional.

To prove estimate (\ref{main_est}), let $\Phi_n$ be the unique solution to 
(\ref{discrete_eq}). According to the argument above, there is a unique 
$\Psi_n \in \Cv_2(S)$ such that 
$\Phi_n = \underline{{\cal L}}_n \Psi_n$ and $\Psi_n$ is the unique solution 
to (\ref{eqPhin}). Now define $\tilde{\Psi}_n$ to be the solution of
\bea
 (\Ima + \Mm^*\Mm + \Mm^* + \Mm + \Jm^*\Jm ) \tilde{\Psi}_n = \underline{{\cal L}}_n  F.
 \label{eqPsin}
\eea
{\us{As a byproduct of Theorem \ref{theorem:well-posedness2}}},
we have that $\tilde{\Psi}_n  \in \Cv_2^{r, \alpha}$
and
\bea
 \|  \tilde{\Psi}_n \|_{C^{r, \alpha}} \leq C   \|\underline{{\cal L}}_n F\|_{\Cv_2^{r, \alpha}}.
\label{PsinCalpha}
\eea
Now, using (\ref{estGan2}) and (\ref{spec2}),
\bea
 \| (\Mm_{n'}  \underline{{\cal L}}_n - \Mm) \tilde{\Psi}_n \|
& \leq & \| \Mm (I -\underline{{\cal L}}_n )\tilde{\Psi}_n \|
+ \| (\Mm_{n'}   - \Mm) \underline{{\cal L}}_n \tilde{\Psi}_n \| \no \\
& \leq & C n^{1/2-r-\alpha} \|\tilde{\Psi}_n  \|_{\Cv_2^{r,\alpha}} .
\label{init}
\eea
Thus,
\bean
\| (\underline{{\cal L}}_n \Mm_{n'}  \underline{{\cal L}}_n -
\underline{{\cal L}}_n \Mm) \tilde{\Psi}_n \| \leq C  
n^{1-r-\alpha} \|\tilde{\Psi}_n  \|_{\Cv_2^{r,\alpha}},
\eean
and as  $\|  (\Ima -\underline{{\cal L}}_n ) \Mm\tilde{\Psi}_n \| \leq C  
n^{1/2-r-\alpha} \|\tilde{\Psi}_n  \|_{\Cv_2^{r,\alpha}}$,
we conclude that
\bea   \label{easy_part}
\| (\underline{{\cal L}}_n \Mm_{n'}  \underline{{\cal L}}_n -
\Mm) \tilde{\Psi}_n \| \leq C  
n^{1-r-\alpha} \|\tilde{\Psi}_n  \|_{\Cv_2^{r,\alpha}}.
\eea
Similarly,
\bea   \label{easy_part2}
\| (\underline{{\cal L}}_n \Mm^*_{n'}  \underline{{\cal L}}_n -
\Mm^*) \tilde{\Psi}_n \| \leq C  
n^{1-r-\alpha} \|\tilde{\Psi}_n  \|_{\Cv_2^{r,\alpha}}.
\eea
From (\ref{init}),
\bea
\| \Mm^*(\Mm_{n'}  \underline{{\cal L}}_n - \Mm) \tilde{\Psi}_n \|_{\Cv_2^\alpha}
\leq C n^{1/2-r - \alpha} \|\tilde{\Psi}_n  \|_{\Cv_2^{r,\alpha}}  \label{KstarKnprimeL} ,
\eea 
and
\bean
\| \Mm^*(I -\underline{{\cal L}}_n ) (\Mm_{n'}  \underline{{\cal L}}_n - \Mm) \tilde{\Psi}_n \|
& \leq & C n^{1-r -  \alpha} \|\tilde{\Psi}_n  \|_{\Cv_2^{r,\alpha}}, \\
\| (\Mm^*_{n'}   - \Mm^*) \underline{{\cal L}}_n 
(\Mm_{n'}  \underline{{\cal L}}_n - \Mm) \tilde{\Psi}_n  \|_{\infty}
& \leq & C n^{1-r - \alpha} \|\tilde{\Psi}_n  \|_{\Cv_2^{r,\alpha}},
\eean
where in the last line we used~(\ref{spec2}).
Thus, using~(\ref{specstar}),
\bea
\| (\Mm^*_{n'} \underline{{\cal L}}_n   - \Mm^*) 
(\Mm_{n'}  \underline{{\cal L}}_n - \Mm) \tilde{\Psi}_n  \|
\leq C n^{1-r - \alpha} \|\tilde{\Psi}_n  \|_{\Cv_2^{r,\alpha}}.
\label{KnprimeLn_etc}
\eea
Now combining~(\ref{KstarKnprimeL}) with~(\ref{KnprimeLn_etc}) yields
\bea
\| (\Mm^*_{n'} \underline{{\cal L}}_n   + \Mm^*) 
(\Mm_{n'}  \underline{{\cal L}}_n - \Mm) \tilde{\Psi}_n  \|
\leq C n^{1-r - \alpha} \|\tilde{\Psi}_n  \|_{\Cv_2^{r,\alpha}}. \label{mixedagain}
\eea
To evaluate $( \Mm^* \Mm_{n'}  \underline{{\cal L}}_n -  
\Mm^*_{n'} \underline{{\cal L}}_n\Mm )\tilde{\Psi}_n$,
write
\bea
 \|( \Mm^* \Mm_{n'}  \underline{{\cal L}}_n -  \Mm^*_{n'} \underline{{\cal L}}_n
\Mm )\tilde{\Psi}_n\|
\leq
 \| \Mm^*( \Mm_{n'} - \Mm)  \underline{{\cal L}}_n \tilde{\Psi}_n\| \no
+ \| \Mm^*(I - \underline{{\cal L}}_n)\Mm \underline{{\cal L}}_n\tilde{\Psi}_n\|\\
+ \| \Mm^* \underline{{\cal L}}_n \Mm (-\Ima + \underline{{\cal L}}_n) \tilde{\Psi}_n\|+
 \| (\Mm^*- \Mm^*_{n'} ) \underline{{\cal L}}_n\Mm \tilde{\Psi}_n \| \no \\
\label{long}
\eea
The first, third and fourth terms in the right hand side
of~(\ref{long})
are clearly bounded by $C n^{1-r - \alpha} \|\tilde{\Psi}_n  \|_{\Cv_2^{r,\alpha}}$.
To tackle the second term in (\ref{long}), we write
\bean
\| \Mm^*(I - \underline{{\cal L}}_n)\Mm \underline{{\cal L}}_n\tilde{\Psi}_n\| \leq
\| \Mm^*(I - \underline{{\cal L}}_n)\Mm \tilde{\Psi}_n\| +
\| \Mm^*(I - \underline{{\cal L}}_n)\Mm (\Ima - \underline{{\cal L}}_n)\tilde{\Psi}_n\|,
\eean
which is again bounded by $C n^{1-r - \alpha} \|\tilde{\Psi}_n  \|$.
Combining~(\ref{mixedagain}) with~(\ref{long}) 
it follows that
\bea
\| 
(\Mm^*_{n'} \underline{{\cal L}}_n \Mm_{n'}  \underline{{\cal L}}_n -
 \Mm^* \Mm) \tilde{\Psi}_n  \|
\leq C n^{1-r - \alpha} \|\tilde{\Psi}_n  \|_{\Cv_2^{r,\alpha}} \label{f1}.
\eea
Similarly, we can show
\bea
\| 
(\Jm^*_{n'} \underline{{\cal L}}_n \Jm_{n'}  
\underline{{\cal L}}_n - \Jm^* \Jm) \tilde{\Psi}_n  \|
\leq  C n^{1-r - \alpha} \|\tilde{\Psi}_n  \|_{\Cv_2^{r,\alpha}}\label{f2}.
\eea
Recalling that $\Psi_n$ satisfies (\ref{eqPhin}) 
and that $\tilde{\Psi}_n$ satisfies (\ref{eqPsin}),
it follows that (using (\ref{easy_part}), (\ref{easy_part2}), (\ref{f1}), (\ref{f2}))
\begin{align*}
  \MoveEqLeft
\| (\Ima + \Mm^*_{n'} \underline{{\cal L}}_n \Mm_{n'}  \underline{{\cal L}}_n 
+ \Mm^*_{n'} \underline{{\cal L}}_n + \Mm_{n'}  \underline{{\cal L}}_n 
+ \Jm^*_{n'} \underline{{\cal L}}_n \Jm_{n'}  
\underline{{\cal L}}_n)  (\Psi_n - \tilde{\Psi}_n ) \| &  \\
& \leq  C n^{1-r - \alpha} \|\tilde{\Psi}_n \|_{\Cv_2^{r,\alpha}}.
\end{align*}
Now, using~(\ref{inv_est}) we obtain 
\bean
\| \Psi_n - \tilde{\Psi}_n \|
\leq  C n^{3/2-r - \alpha} \|\tilde{\Psi}_n \|_{\Cv_2^{r,\alpha}}
. \eean
It follows from equations~(\ref{Phieq}) and~(\ref{eqPsin})
that 
$$ \| \Phi - \tilde{\Psi}_n  \| \leq C \| (\Ima - \underline{{\cal L}}_n  ) F \| 
\leq C n^{1/2-r - \alpha} \| F \|_{\Cv_2^{r,\alpha}}. $$
It then follows from estimate~(\ref{PsinCalpha})
that
\bea
\| \Psi_n - \Phi \|
\leq  C n^{2-r - \alpha} \| F \|_{\Cv_2^{r,\alpha}}. \label{PsinPhi}
\eea
Finally we recall that $\Phi_n = \underline{{\cal L}}_n \Psi_n $
and we use (\ref{PsinPhi}) and (\ref{estGan1}) to write
\bean
\| \Phi_n - \Phi \| \leq \| \underline{{\cal L}}_n \Psi_n - \underline{{\cal L}}_n\Phi \|+
\|(\Ima- \underline{{\cal L}}_n) \Phi) \| \leq  C n^{5/2-r - \alpha} \| F \|_{\Cv_2^{r,\alpha}} ,
\eean
so (\ref{main_est}) is proved.
\end{proof}

\section{Numerical experiments}
\label{sec:numerics}

To demonstrate our algorithm we make use of a discretization that
preserves adjointness. To that end,
following details in~\cite{dielectric2}
we write the component
integral operators of $\Mm$ and $\Jm$, defined in~\eqref{eq:Mm}
and~\eqref{J-op} respectively, in the form
\begin{equation}
  \label{eq:a}
  \f(\x) = \int_{\surface} m(\x,\y) \w(\y) \; ds(\y),
\end{equation}
where the kernel $m(\x,\y)$ is a $3 \times 3$ matrix and $\w(\y),\f(\x)$
are column $3$-vectors.
As discussed earlier, we assume that the surface $\surface$ can be parametrized
by a diffeomorphic map  $\q: \sphere \to \surface$,
\begin{equation}
  \label{eq:qmap}
  \x = \q(\xh), \qquad \x \in \surface, \quad\ \xh \in \sphere.
\end{equation}
Our discretization of~\eqref{eq:a} uses the 
transform back to the unit sphere $\sphere$,
\begin{equation}
  \label{eq:b}
  \F(\xh) = \int_{\sphere} M(\xh,\yh) \W(\yh) \; J(\yh) \; ds(\yh),
\end{equation}
where $J$ is the Jacobian of $\q$ and
\begin{align}
  \label{eq:pullback}
  \begin{split}
  M(\xh,\yh) & = J(\xh)^{1/2} \; m(\q(\xh),\q(\yh)) \; J(\yh)^{1/2},\\
  \W(\xh) & = J(\xh)^{1/2} \; \w(\q(\xh)), \\
  \F(\xh) & = J(\xh)^{1/2} \; \f(\q(\xh)).
  \end{split}
\end{align}
The far-field $\E_\infty$ is obtained similarly using surface integral
representations~\cite[Section~3]{dielectric2}.
{\revA{
More precisely,
we use $\Phi_n$ and the exact surface integral representation
of the  electric far-field $\E^\infty$  in~\cite[Equation~(6.2)]{ganesh2014all}
to compute the spectrally accurate approximation  $\E^\infty_n$.
Since  the far-field kernels are smooth, following arguments in the final part of the proof in~\cite[Theorem~3]{gh:first}, $\E^\infty_n$ retains the same
 convergence properties with respect to $n$ as $\Phi_n$. }}
The discretization described in Section~\ref{sec:spectral},
which uses a basis of spherical harmonics
defined on the spherical reference domain $\sphere$,
then preserves adjointness of the transformed surface integral operators,
so that the matrices
representing $\Mm^*$ and $\Jm^*$ are readily computed as the conjugate
transpose of the matrices representing $\Mm$ and $\Jm$.
In the remainder of this section we use the subscript $n$ to denote
approximations computed using the fully-discrete Galerkin scheme
with spherical harmonics of degree not greater than $n$.
Thus, for example, $\E^{\infty}_{n}$ is our approximation to $\E^\infty$.

To demonstrate our algorithm
we present numerical results for scattering of a plane wave
\begin{equation}
  \label{eq:plane-wave}
  \E_\inc(\x) = \p \, e^{i k \, \x \cdot \d},
\end{equation}
by a  sphere,  spheroids and a non-convex Chebyshev particle, which are
common test geometries in the literature~\cite{mish94,kahnert,rother:cheby}
and of interest in applications such as atmospheric science. (The sphere case facilitates
comparison with the analytical Mie series solution~\cite{wiscombe:mie}.)
Here $\d$ is the incident wave direction, and $\p$ satisfying $\d \cdot \p = 0$
is the incident polarization.
The spheroid with aspect ratio $\rho$ is parametrized by
\begin{equation}
  \label{eq:spheroid}
  \q(\xh) = \left( \begin{array}{c}  \xh_1 / \rho\\ \xh_2 / \rho \\ \xh_3
    \end{array} \right), \qquad \xh \in \sphere,
\end{equation}
and is oriented so that its major axis aligns with the $z$-axis.
The Chebyshev particle is parametrized by~\cite{rother:cheby}
\begin{equation}
  \label{eq:chebyshev}
  \q(\xh) = r(\xh) \, \xh,
  \qquad r(\x) = \frac{1}{2} + \frac{1}{40} \cos (5 \cos^{-1} \xh_3).
  \qquad \xh \in \sphere,
\end{equation}
This particle is visualized in
Figure~\ref{fig:ext-hive8-mishchenko2}.
In our experiments we use dielectric bodies with
real refractive index $\nu = 1.584$, which is the refractive index
of polycarbonate, and based on numerical results in~\cite{mish94},
absorbing bodies with
complex refractive indices
$\nu = 1.5+0.02i, \ 1.0925 + 0.248i$.
Thus we demonstrate our algorithm for non-absorbing, weakly-absorbing
and strongly-absorbing particles.
The electromagnetic size of the particles is
$s = d/\lambda$ where $d$ denotes the diameter of the particle,
and the commonly used size parameter is therefore $x= \pi s$.

To simplify the exposition, except where otherwise stated,
we compute the near field $\E$ and the
far field $\E^\infty_n$ in the $xz$-plane for incident waves with
direction $\d = (0,0,1)^T$.
Thus the scattering plane is the $xz$-plane and,
following~\cite[Page~245]{rother},~\cite[Page~387]{hulst:light},
and~\cite[Figure~4.5, Page~204]{kristensson}
we define V-polarization and H-polarization as being
respectively perpendicular and parallel to the scattering plane.
In particular,
for a receiver in the $xz$-plane
in the direction $\xh = (\sin \theta,0,\cos \theta)^T$,
we define the vertical and horizontal polarization
directions respectively to be
$\e_V = \e_\phi = (0,1,0)^T$
and
$\e_H = \e_\theta = (\cos \theta,0,-\sin \theta)^T$.
The radar cross section of the scatterer, measured in the scattering plane,
is then
\begin{equation}
  \label{eq:rcs}
  \sigma_{\mathrm{X} \mathrm{X}}(\theta) = 10 \log_{10} 4 \pi |\e_\mathrm{X}(\theta) \cdot \E_\infty(\theta)|^2,
\end{equation}
where $\mathrm{X}=\mathrm{H}$ or $\mathrm{X}=\mathrm{V}$ denotes the receiver polarization, and the
incident wave polarization is chosen (H or V)
to match the receiver polarization.
The scattering plane and polarization are visualized in
Figure~\ref{fig:scattering-plane}.

\begin{figure}
  \centering
  \begin{tikzpicture}[scale=0.6]
    \draw (1.24,-0.00) -- (1.24,0.08);
\draw (1.24,0.08) -- (1.23,0.16);
\draw (1.23,0.16) -- (1.22,0.24);
\draw (1.22,0.24) -- (1.20,0.31);
\draw (1.20,0.31) -- (1.18,0.39);
\draw (1.18,0.39) -- (1.15,0.46);
\draw (1.15,0.46) -- (1.12,0.53);
\draw (1.12,0.53) -- (1.09,0.60);
\draw (1.09,0.60) -- (1.05,0.67);
\draw (1.05,0.67) -- (1.00,0.74);
\draw (1.00,0.74) -- (0.95,0.80);
\draw (0.95,0.80) -- (0.90,0.86);
\draw (0.90,0.86) -- (0.84,0.91);
\draw (0.84,0.91) -- (0.78,0.96);
\draw (0.78,0.96) -- (0.72,1.01);
\draw (0.72,1.01) -- (0.65,1.06);
\draw (0.65,1.06) -- (0.59,1.09);
\draw (0.59,1.09) -- (0.52,1.13);
\draw (0.52,1.13) -- (0.44,1.16);
\draw (0.44,1.16) -- (0.37,1.19);
\draw (0.37,1.19) -- (0.29,1.21);
\draw (0.29,1.21) -- (0.22,1.22);
\draw (0.22,1.22) -- (0.14,1.23);
\draw (0.14,1.23) -- (0.06,1.24);
\draw (0.06,1.24) -- (-0.02,1.24);
\draw (-0.02,1.24) -- (-0.10,1.24);
\draw (-0.10,1.24) -- (-0.18,1.23);
\draw (-0.18,1.23) -- (-0.25,1.22);
\draw (-0.25,1.22) -- (-0.33,1.20);
\draw (-0.33,1.20) -- (-0.41,1.17);
\draw (-0.41,1.17) -- (-0.48,1.15);
\draw (-0.48,1.15) -- (-0.55,1.11);
\draw (-0.55,1.11) -- (-0.62,1.08);
\draw (-0.62,1.08) -- (-0.69,1.03);
\draw (-0.69,1.03) -- (-0.75,0.99);
\draw (-0.75,0.99) -- (-0.81,0.94);
\draw (-0.81,0.94) -- (-0.87,0.89);
\draw (-0.87,0.89) -- (-0.93,0.83);
\draw (-0.93,0.83) -- (-0.98,0.77);
\draw (-0.98,0.77) -- (-1.02,0.70);
\draw (-1.02,0.70) -- (-1.07,0.64);
\draw (-1.07,0.64) -- (-1.10,0.57);
\draw (-1.10,0.57) -- (-1.14,0.50);
\draw (-1.14,0.50) -- (-1.17,0.42);
\draw (-1.17,0.42) -- (-1.19,0.35);
\draw (-1.19,0.35) -- (-1.21,0.27);
\draw (-1.21,0.27) -- (-1.23,0.20);
\draw (-1.23,0.20) -- (-1.24,0.12);
\draw (-1.24,0.12) -- (-1.24,0.04);
\draw (-1.24,0.04) -- (-1.24,-0.04);
\draw (-1.24,-0.04) -- (-1.24,-0.12);
\draw (-1.24,-0.12) -- (-1.23,-0.20);
\draw (-1.23,-0.20) -- (-1.21,-0.27);
\draw (-1.21,-0.27) -- (-1.19,-0.35);
\draw (-1.19,-0.35) -- (-1.17,-0.42);
\draw (-1.17,-0.42) -- (-1.14,-0.50);
\draw (-1.14,-0.50) -- (-1.10,-0.57);
\draw (-1.10,-0.57) -- (-1.07,-0.64);
\draw (-1.07,-0.64) -- (-1.02,-0.70);
\draw (-1.02,-0.70) -- (-0.98,-0.77);
\draw (-0.98,-0.77) -- (-0.93,-0.83);
\draw (-0.93,-0.83) -- (-0.87,-0.89);
\draw (-0.87,-0.89) -- (-0.81,-0.94);
\draw (-0.81,-0.94) -- (-0.75,-0.99);
\draw (-0.75,-0.99) -- (-0.69,-1.03);
\draw (-0.69,-1.03) -- (-0.62,-1.08);
\draw (-0.62,-1.08) -- (-0.55,-1.11);
\draw (-0.55,-1.11) -- (-0.48,-1.15);
\draw (-0.48,-1.15) -- (-0.41,-1.17);
\draw (-0.41,-1.17) -- (-0.33,-1.20);
\draw (-0.33,-1.20) -- (-0.25,-1.22);
\draw (-0.25,-1.22) -- (-0.18,-1.23);
\draw (-0.18,-1.23) -- (-0.10,-1.24);
\draw (-0.10,-1.24) -- (-0.02,-1.24);
\draw (-0.02,-1.24) -- (0.06,-1.24);
\draw (0.06,-1.24) -- (0.14,-1.23);
\draw (0.14,-1.23) -- (0.22,-1.22);
\draw (0.22,-1.22) -- (0.29,-1.21);
\draw (0.29,-1.21) -- (0.37,-1.19);
\draw (0.37,-1.19) -- (0.44,-1.16);
\draw (0.44,-1.16) -- (0.52,-1.13);
\draw (0.52,-1.13) -- (0.59,-1.09);
\draw (0.59,-1.09) -- (0.65,-1.06);
\draw (0.65,-1.06) -- (0.72,-1.01);
\draw (0.72,-1.01) -- (0.78,-0.96);
\draw (0.78,-0.96) -- (0.84,-0.91);
\draw (0.84,-0.91) -- (0.90,-0.86);
\draw (0.90,-0.86) -- (0.95,-0.80);
\draw (0.95,-0.80) -- (1.00,-0.74);
\draw (1.00,-0.74) -- (1.05,-0.67);
\draw (1.05,-0.67) -- (1.09,-0.60);
\draw (1.09,-0.60) -- (1.12,-0.53);
\draw (1.12,-0.53) -- (1.15,-0.46);
\draw (1.15,-0.46) -- (1.18,-0.39);
\draw (1.18,-0.39) -- (1.20,-0.31);
\draw (1.20,-0.31) -- (1.22,-0.24);
\draw (1.22,-0.24) -- (1.23,-0.16);
\draw (1.23,-0.16) -- (1.24,-0.08);
\draw (1.24,-0.08) -- (1.24,-0.00);
\draw (0.59,0.48) -- (0.51,0.49);
\draw (0.51,0.49) -- (0.43,0.50);
\draw (0.43,0.50) -- (0.35,0.51);
\draw (0.35,0.51) -- (0.27,0.52);
\draw (0.27,0.52) -- (0.19,0.53);
\draw (0.19,0.53) -- (0.11,0.53);
\draw (0.11,0.53) -- (0.03,0.53);
\draw (0.03,0.53) -- (-0.06,0.53);
\draw (-0.06,0.53) -- (-0.14,0.53);
\draw (-0.14,0.53) -- (-0.22,0.53);
\draw (-0.22,0.53) -- (-0.30,0.52);
\draw (-0.30,0.52) -- (-0.38,0.51);
\draw (-0.38,0.51) -- (-0.46,0.50);
\draw (-0.46,0.50) -- (-0.54,0.49);
\draw (-0.54,0.49) -- (-0.62,0.47);
\draw (-0.62,0.47) -- (-0.69,0.46);
\draw (-0.69,0.46) -- (-0.76,0.44);
\draw (-0.76,0.44) -- (-0.83,0.42);
\draw (-0.83,0.42) -- (-0.89,0.40);
\draw (-0.89,0.40) -- (-0.95,0.37);
\draw (-0.95,0.37) -- (-1.01,0.35);
\draw (-1.01,0.35) -- (-1.07,0.32);
\draw (-1.07,0.32) -- (-1.12,0.30);
\draw (-1.12,0.30) -- (-1.16,0.27);
\draw (-1.16,0.27) -- (-1.21,0.24);
\draw (-1.21,0.24) -- (-1.24,0.21);
\draw (-1.24,0.21) -- (-1.27,0.17);
\draw (-1.27,0.17) -- (-1.30,0.14);
\draw (-1.30,0.14) -- (-1.32,0.11);
\draw (-1.32,0.11) -- (-1.34,0.07);
\draw (-1.34,0.07) -- (-1.35,0.04);
\draw (-1.35,0.04) -- (-1.36,0.00);
\draw (-1.36,0.00) -- (-1.36,-0.03);
\draw (-1.36,-0.03) -- (-1.36,-0.07);
\draw (-1.36,-0.07) -- (-1.35,-0.10);
\draw (-1.35,-0.10) -- (-1.33,-0.14);
\draw (-1.33,-0.14) -- (-1.31,-0.17);
\draw (-1.31,-0.17) -- (-1.28,-0.21);
\draw (-1.28,-0.21) -- (-1.25,-0.24);
\draw (-1.25,-0.24) -- (-1.22,-0.27);
\draw (-1.22,-0.27) -- (-1.17,-0.31);
\draw (-1.17,-0.31) -- (-1.13,-0.34);
\draw (-1.13,-0.34) -- (-1.07,-0.37);
\draw (-1.07,-0.37) -- (-1.02,-0.39);
\draw (-1.02,-0.39) -- (-0.96,-0.42);
\draw (-0.96,-0.42) -- (-0.89,-0.44);
\draw (-0.89,-0.44) -- (-0.82,-0.47);
\draw (-0.82,-0.47) -- (-0.75,-0.49);
\draw (-0.75,-0.49) -- (-0.67,-0.51);
\draw (-0.67,-0.51) -- (-0.59,-0.52);
\draw (-0.59,-0.52) -- (-0.51,-0.54);
\draw (-0.51,-0.54) -- (-0.43,-0.55);
\draw (-0.43,-0.55) -- (-0.34,-0.56);
\draw (-0.34,-0.56) -- (-0.25,-0.57);
\draw (-0.25,-0.57) -- (-0.16,-0.58);
\draw (-0.16,-0.58) -- (-0.07,-0.58);
\draw (-0.07,-0.58) -- (0.02,-0.58);
\draw (0.02,-0.58) -- (0.11,-0.58);
\draw (0.11,-0.58) -- (0.20,-0.57);
\draw (0.20,-0.57) -- (0.29,-0.57);
\draw (0.29,-0.57) -- (0.37,-0.56);
\draw (0.37,-0.56) -- (0.46,-0.55);
\draw (0.46,-0.55) -- (0.54,-0.53);
\draw (0.54,-0.53) -- (0.62,-0.52);
\draw (0.62,-0.52) -- (0.70,-0.50);
\draw (0.70,-0.50) -- (0.78,-0.48);
\draw (0.78,-0.48) -- (0.85,-0.46);
\draw (0.85,-0.46) -- (0.92,-0.43);
\draw (0.92,-0.43) -- (0.98,-0.41);
\draw (0.98,-0.41) -- (1.04,-0.38);
\draw (1.04,-0.38) -- (1.09,-0.35);
\draw (1.09,-0.35) -- (1.15,-0.32);
\draw (1.15,-0.32) -- (1.19,-0.29);
\draw (1.19,-0.29) -- (1.23,-0.26);
\draw (1.23,-0.26) -- (1.27,-0.23);
\draw (1.27,-0.23) -- (1.30,-0.19);
\draw (1.30,-0.19) -- (1.32,-0.16);
\draw (1.32,-0.16) -- (1.34,-0.13);
\draw (1.34,-0.13) -- (1.35,-0.09);
\draw (1.35,-0.09) -- (1.36,-0.05);
\draw (1.36,-0.05) -- (1.36,-0.02);
\draw (1.36,-0.02) -- (1.36,0.02);
\draw (1.36,0.02) -- (1.35,0.05);
\draw (1.35,0.05) -- (1.34,0.09);
\draw (1.34,0.09) -- (1.32,0.12);
\draw (1.32,0.12) -- (1.29,0.15);
\draw (1.29,0.15) -- (1.26,0.19);
\draw (1.26,0.19) -- (1.23,0.22);
\draw (1.23,0.22) -- (1.19,0.25);
\draw (1.19,0.25) -- (1.15,0.28);
\draw (1.15,0.28) -- (1.10,0.31);
\draw (1.10,0.31) -- (1.05,0.33);
\draw (1.05,0.33) -- (0.99,0.36);
\draw (0.99,0.36) -- (0.93,0.38);
\draw (0.93,0.38) -- (0.87,0.41);
\draw (0.87,0.41) -- (0.80,0.43);
\draw (0.80,0.43) -- (0.73,0.45);
\draw (0.73,0.45) -- (0.66,0.46);
\draw (0.66,0.46) -- (0.59,0.48);
\draw (4.21,-5.16) -- (8.36,1.14);
\draw (8.36,1.14) -- (-2.38,2.92);
\draw (-2.38,2.92) -- (-10.06,-1.37);
\draw (-10.06,-1.37) -- (4.21,-5.16);
\draw (0.00,-0.00) -- (-6.32,-5.16);
\draw[->] (-6.32,-5.16) -- (-5.21,-4.25);
\draw[->] (-6.32,-5.16) -- (-4.76,-5.65);
\draw[->] (-6.32,-5.16) -- (-6.49,-3.53);
\draw (-6.49,-3.53) node[anchor=south] {$\e_V(\pi)$};
\draw (-4.76,-5.65) node[anchor=west] {$\e_H(\pi)$};
\draw (-5.21,-4.25) node[anchor=north] {$\boldsymbol{d}$};
\draw (6.81,3.26) node[anchor=south] {$\e_V(\theta)$};
\draw (5.83,2.47) node[anchor=east] {$\e_H(\theta)$};
\draw (7.43,2.41) node[anchor=west] {$\xh(\theta)$};
\draw (0.00,-0.00) -- (6.70,2.17);
\draw[->] (6.70,2.17) -- (6.81,3.26);
\draw[->] (6.70,2.17) -- (5.83,2.47);
\draw[->] (6.70,2.17) -- (7.43,2.41);
\fill (0.00,-0.00) circle[radius=1.36];
\draw (0.00,-0.00) -- (2.57,2.10);
\draw (3.65,1.18) -- (3.62,1.19);
\draw (3.62,1.19) -- (3.59,1.21);
\draw (3.59,1.21) -- (3.56,1.22);
\draw (3.56,1.22) -- (3.53,1.23);
\draw (3.53,1.23) -- (3.50,1.24);
\draw (3.50,1.24) -- (3.47,1.25);
\draw (3.47,1.25) -- (3.44,1.26);
\draw (3.44,1.26) -- (3.40,1.27);
\draw (3.40,1.27) -- (3.37,1.29);
\draw (3.37,1.29) -- (3.34,1.30);
\draw (3.34,1.30) -- (3.31,1.31);
\draw (3.31,1.31) -- (3.28,1.32);
\draw (3.28,1.32) -- (3.25,1.33);
\draw (3.25,1.33) -- (3.22,1.34);
\draw (3.22,1.34) -- (3.18,1.35);
\draw (3.18,1.35) -- (3.15,1.36);
\draw (3.15,1.36) -- (3.12,1.37);
\draw (3.12,1.37) -- (3.09,1.38);
\draw (3.09,1.38) -- (3.05,1.39);
\draw (3.05,1.39) -- (3.02,1.40);
\draw (3.02,1.40) -- (2.99,1.41);
\draw (2.99,1.41) -- (2.95,1.42);
\draw (2.95,1.42) -- (2.92,1.43);
\draw (2.92,1.43) -- (2.89,1.44);
\draw (2.89,1.44) -- (2.85,1.45);
\draw (2.85,1.45) -- (2.82,1.46);
\draw (2.82,1.46) -- (2.79,1.47);
\draw (2.79,1.47) -- (2.75,1.48);
\draw (2.75,1.48) -- (2.72,1.48);
\draw (2.72,1.48) -- (2.68,1.49);
\draw (2.68,1.49) -- (2.65,1.50);
\draw (2.65,1.50) -- (2.61,1.51);
\draw (2.61,1.51) -- (2.58,1.52);
\draw (2.58,1.52) -- (2.54,1.53);
\draw (2.54,1.53) -- (2.51,1.53);
\draw (2.51,1.53) -- (2.47,1.54);
\draw (2.47,1.54) -- (2.44,1.55);
\draw (2.44,1.55) -- (2.40,1.56);
\draw (2.40,1.56) -- (2.37,1.57);
\draw (2.37,1.57) -- (2.33,1.57);
\draw (2.33,1.57) -- (2.29,1.58);
\draw (2.29,1.58) -- (2.26,1.59);
\draw (2.26,1.59) -- (2.22,1.60);
\draw (2.22,1.60) -- (2.19,1.60);
\draw (2.19,1.60) -- (2.15,1.61);
\draw (2.15,1.61) -- (2.11,1.62);
\draw (2.11,1.62) -- (2.08,1.62);
\draw (2.08,1.62) -- (2.04,1.63);
\draw (2.04,1.63) -- (2.00,1.64);
\draw (2.21,1.11) node {$\theta$};
  \end{tikzpicture}
  \caption{\label{fig:scattering-plane}
    Schematic showing vertical (V) and horizontal (H) polarization
    with respect to the scattering plane
    for an incident wave with direction $\d$ and a receiver in the direction
    $\xh(\theta)$.}
\end{figure}
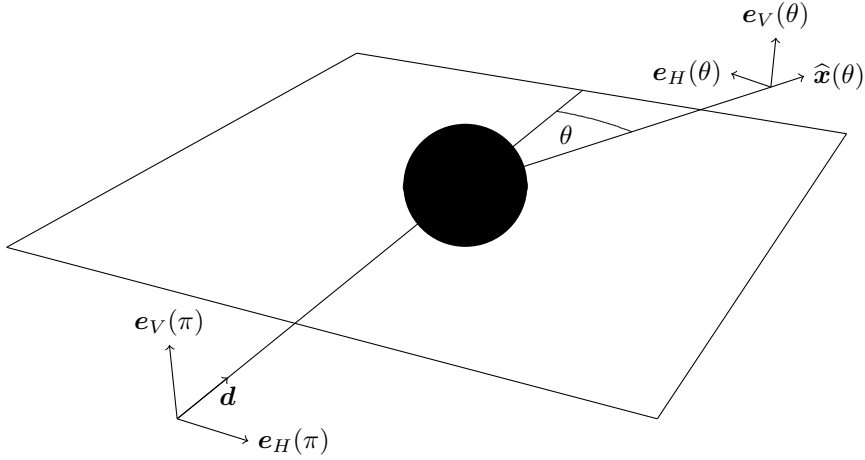

\begin{figure}
  \centering
  \includegraphics[width=10cm]{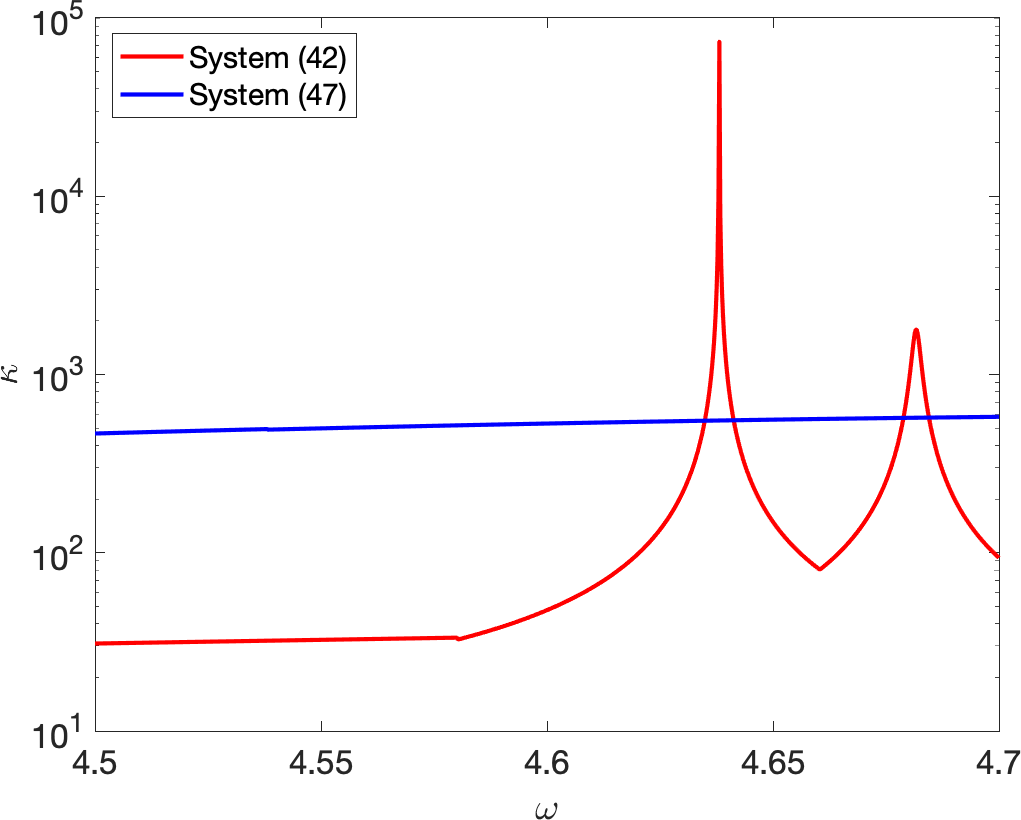}
  \caption{\label{fig:cond}
    Condition number $\kappa$ of the matrix in systems~\eqref{eq:strawman}
    (without stabilization) and~\eqref{eq:SIE} (with stabilization)
    for a dielectric unit sphere with
    refractive index $\nu = 1.303324617806156$ and
    frequency $\omega$.
    }
\end{figure}

In corroboration of the robust mathematical analysis
in Section~\ref{sec:J}  proving all-frequency stability and well-posedness of the continuous system, we begin by numerically demonstrating 
the effectiveness of the
fully-discrete counterpart of our new formulation~\eqref{eq:SIE}, governed by second-kind self-adjoint operator $I + \widetilde{\Mm}$, 
compared with the related system~\eqref{eq:strawman}, governed by  $I + \Mm$ (both without further external constraints).
To this end, in Figure~\ref{fig:cond} we plot
the 1-norm condition number of
the matrices representing both operators
(obtained using the fully-discrete representation described above)
for a unit sphere dielectric scatterer
with refractive index $\nu = 1.303324617806156$,
(for which $I + \Mm$
has an eigenfrequency $\omega = 4.638138138138138$).
The plot  was obtained using $1000$ frequencies $\omega$, which involves solving the
Maxwell equations~\eqref{eq:reduced-maxwell} more than $1000$ times
(with fixed dielectric constants and varying incident wave frequencies). 
  The condition number is estimated in both cases using~\cite{higham}.

To demonstrate the correctness of our fully-discrete algorithm for the all-frequency stable system~\eqref{eq:SIE},
\revC{in
  Tables~\ref{table:sph0p000001}--\ref{table:sph8}} we tabulate
the relative error
\begin{equation}
  \label{eq:err-mie}
  \mathrm{ERR}_{\mathrm{Mie},n} =
  \frac{\| \E_{\infty,n}(\theta) - \E_{\infty,\mathrm{ref}}(\theta) \|_\infty}{\| \E_{\infty,\mathrm{ref}}(\theta) \|_\infty}
\end{equation}
for scattering of H-polarized incident plane waves
by the unit sphere.
The reference solution $\E_{\infty,\mathrm{ref}}$ is computed to
high precision using
Wiscombe's Mie series code~\cite{wiscombe:mie}.
In practice, the maximum norms in~\eqref{eq:err-mie} are approximated discretely
using more than 1200 evenly spaced points $\theta \in [0,2\pi]$.

For non-spherical scatterers we demonstrate the correctness of our
algorithm using the reciprocity relation~\cite[Theorem~6.30]{colton:inverse}.
In particular, in Tables~\ref{table:ell:1}--\ref{table:hive8}
we tabulate the relative error
\begin{equation}
  \label{eq:err-rp}
  \mathrm{ERR}_{\mathrm{RP},n} =
  \frac{\| \e_H(\theta) \cdot \E_{\infty,n}(\theta;\theta') - \e_H(\theta') \cdot \E_{\infty,n}(\theta';\theta) \|_\infty}{\|  (\E_{\infty,n}(\theta;\theta') \cdot \E_{\infty,n}(\theta';\theta))^{1/2} \|_\infty},
\end{equation}
which measures the residual in the reciprocity relation.
In~\eqref{eq:err-rp}
we denote by $\E_{\infty,n}(\cdot;\theta')$ the far field induced by the
$H$-polarized incident wave with direction
$\d = -(\sin \theta',0,\cos \theta')^T$.
The norms in~\eqref{eq:err-rp} are over $(\theta,\theta') \in
[0,2 \pi] \times [0,2 \pi]$ and are approximated discretely using more than
$1200 \times 1200$ points.
The tabulated results
demonstrate the superalgebraic convergence
established by
the numerical analysis in~\eqref{main_est}.
In particular,
high order accuracy is attained with only a small increase in the
discretization parameter $n$
(with $r$ in~\eqref{main_est} sufficiently large due to smoothness
of the test particles).
In Table~\ref{table:cpu}
we tabulate  the CPU time  and memory for computing
the non-convex penetrable Chebyshev particle results in Tables~\ref{table:hive1}--\ref{table:hive8}, demonstrating that the
spectral accuracy of our fully-discrete algorithm
facilitates simulation of the low to medium-frequency Maxwell model
in three dimensions
to high-accuracy even using laptops with multicore
processors.

For convenient visualization of the reciprocity relation, in
Figures~\ref{fig:far-hive8-polycarbonate}--\ref{fig:far-hive8-mishchenko2}
we plot the radar cross section $\sigma_{\mathrm{H}\mathrm{H}}$
  of the Chebyshev particle
in its original orientation, and rotated by an angle $\pi/2$ clockwise 
about the
$y$-axis.
{\revA{Following~\cite[Page~314]{rother}}}, the coincidence of the 
radar cross sections at $\theta = \pi/2$ demonstrates that the computed
solutions satisfy
the reciprocity relation (using
the symmetry of the Chebyshev particle about the $xy$-plane).

In Figure~\ref{fig:mish-ell2-size8}
we demonstrate the visual agreement of the radar cross section 
$\sigma_{\mathrm{H}\mathrm{H}}$
  computed using our code and
  using Mishchenko's null-field-method code
  (downloaded from~\cite{mish:code} and documented in~\cite{mish2000}) for
spheroids with aspect ratio $\rho = 2$ and electromagnetic size $s=8$.
  In Figure~\ref{fig:mish-ell5-size1} we demonstrate similar
  agreement for spheroids with aspect ratio $\rho = 5$
  and electromagnetic size $s=1$.
  The null field method is not able to compute the radar cross section
  for spheroids with such large aspect ratio $\rho = 5$
  and electromagnetic size $s=8$ because it
  exhibits numerical instability for large
  aspect ratio scatterers and large electromagnetic sizes~\cite{mish94}.
  The target accuracy of the computations using Mishchenko's code is about
  $10^{-3}$.

  \revB{%
    The induced electric field $\E^{\pm}$ is given by the
    surface integral representation~\cite[Equation~(6.1)]{ganesh2014all}
    (see also \cite[Equations~(2.8)--(2.9)]{dielectric2})
    and computed in a similar way to the far-field.
 In Figure~\ref{fig:ext-hive8-mishchenko2} we visualize
  the total field
  $|\mathop{\mathrm{Re}} \E_{\tot,n}(\x)|$
  and the relative error in the induced electric field
  for Chebyshev particles with electromagnetic size $s=1$ and $s=8$
  illuminated by an H-polarized incident plane wave.
  Because the true induced field is not known for this geometry, the
  relative error at $\x$ is approximated by
  \begin{displaymath}
    \mathrm{ERR}_{\pm,n}(\x)= 
    \frac{| \E^{\pm}_n(\x) - \E^{\pm}_{n+10}(\x) | }{|\E^{\pm}_{n+10}(\x)|}
  \end{displaymath}
  where $\E^{\pm}_n$ denotes the computed induced electric field
  in the exterior and interior respectively,
  and we use the
  solution computed with increased parameter $n+10$ as the reference
  solution.
  As $\x$ approaches the boundary $\surface$,  the accuracy is compromised due to the near-singularity in the kernel of the near-field  potentials  associated with surface currents. 
 To ensure high-order accuracy in the numerical evaluation of near-fields very close to the boundary, it is necessary to develop specialized techniques for accurately 
 evaluating these near-singular surface potential integrals. 
 Techniques such as those described in~\cite{helsing, near-sing-pap} can be used in conjunction with our high-order accurate surface currents. 
 However, we have not yet implemented these sophisticated techniques for near-field evaluations. Using our own approach  with approximations induced by the operator
 $ \underline{{\cal L}}_n$ with high-order convergence property~\eqref{estGan2},  as observed in Figure~\ref{fig:ext-hive8-mishchenko2},
 we obtain about five-digits accuracy even very close to the boundary. 
   }

  \revB{We close the numerical experiments with results 
    for plasmonic and  eddy current (very low frequency)
    material parameters. We recall that our continuous and discrete model analyses 
    are not applicable for plasmonic material parameters.    In Tables~\ref{table:eddy0p00001}--\ref{table:eddy0p0001}
    we explore the eddy current regime
    considered in~\cite{helsing:eddy}
    by tabulating the relative
    error~\eqref{eq:err-mie}
    for unit spheres with refractive indices having modulus
    between $10^2$ and $10^4$, and argument $\pi/4$.
    For these results we use our own Mie code.
    In Figure~\ref{fig:cond-eddy} we plot the 1-norm condition number
    of the matrix representing $I + \widetilde{\Mm}$
    for a unit sphere scatterer with material properties described by
    \begin{equation}
      \label{eq:x}
      x = \frac{1+\widehat{\epsilon}}{1-\widehat{\epsilon}},
    \end{equation}
    where $\widehat{\epsilon} = \epsilon_- / \epsilon_+$ is the
    relative permittivity (which, in our examples, is equal to the square of the refractive index $\nu$). 
    The
    condition number is computed for 
    refractive indices on the path
    shown in Figure~\ref{fig:path} using the same methodology used for
    Figure~\ref{fig:cond}.
    The same path is 
    used in~\cite[Figure~10]{helsing}, and
    the eddy current region is characterised by $x$ near 1,
    whilst the plasmonic regime is characterised by
    $x \in (-1,1)$.
    Finally, in Figure~\ref{fig:ext-hive2-plasmonic}
    we visualize the total field 
    $|\mathop{\mathrm{Re}} \E_{\tot,n}(\x)|$
    for a Chebyshev particle with electromagnetic size $s=2$
    illuminated by an H-polarized incident plane wave.
    The refractive index is
    $\nu = 1.088025734989757i$, so that the relative permittivity
    is negative, leading to plasmonic surface waves.
  }
   
  \begin{mietable}[\revisedC]{$10^{-6}\lambda$}{table:sph0p000001}
1  &  1.00e+00  &  9.44e-01  &  7.06e-01 \\
2  &  4.21e-10  &  4.27e-10  &  4.44e-10 \\
3  &  4.13e-10  &  4.88e-10  &  6.94e-10 \\
\end{mietable}

  \begin{mietable}[\revisedC]{$10^{-3}\lambda$}{table:sph0p001}
1  &  1.00e+00  &  9.44e-01  &  7.06e-01 \\
2  &  6.92e-07  &  6.64e-07  &  5.56e-07 \\
3  &  1.01e-12  &  9.36e-13  &  8.26e-13 \\
\end{mietable}

\begin{mietable}{$\lambda$}{table:sph1}
5  &  8.04e-03  &  7.35e-03  &  6.72e-03 \\
10  &  1.25e-09  &  1.17e-09  &  1.26e-09 \\
15  &  1.03e-09  &  1.02e-09  &  1.61e-09 \\
\end{mietable}

\begin{mietable}{$8 \lambda$}{table:sph8}
30  &  1.46e-01  &  1.27e-01  &  5.79e-04 \\
35  &  1.91e-03  &  1.59e-04  &  3.15e-07 \\
40  &  1.98e-08  &  2.18e-10  &  5.65e-11 \\
\end{mietable}

\begin{rptable}{spheroid}{$\lambda$}{table:ell:1}
  \multicolumn{2}{l}{Aspect ratio 2}\\
  \cline{1-3}
5  &  1.61e-03  &  1.44e-03  &  1.52e-03 \\
10  &  3.76e-07  &  3.37e-07  &  6.31e-07 \\
15  &  1.39e-09  &  1.30e-09  &  1.45e-09 \\
  \multicolumn{2}{l}{Aspect ratio 3}\\
  \cline{1-3}
5  &  1.61e-03  &  1.27e-03  &  1.74e-03 \\
15  &  2.84e-07  &  2.87e-07  &  3.66e-07 \\
25  &  1.52e-10  &  1.54e-10  &  1.97e-10 \\
  \multicolumn{2}{l}{Aspect ratio 4}\\
  \cline{1-3}
5  &  2.00e-03  &  1.64e-03  &  2.72e-03 \\
15  &  2.95e-06  &  3.14e-06  &  4.81e-06 \\
25  &  9.49e-09  &  1.01e-08  &  1.55e-08 \\
  \multicolumn{2}{l}{Aspect ratio 5}\\
  \cline{1-3}
10  &  9.93e-05  &  1.33e-04  &  5.27e-04 \\
20  &  5.23e-07  &  7.24e-07  &  3.32e-06 \\
30  &  4.16e-09  &  6.03e-09  &  3.21e-08 \\
\end{rptable}

\begin{rptable}{spheroid}{$8 \lambda$}{table:ell:8}
  \multicolumn{2}{l}{Aspect ratio 2}\\
  \cline{1-3}
20  &  4.62e-01  &  3.77e-01  &  7.99e-01 \\
30  &  6.33e-02  &  1.76e-02  &  8.53e-04 \\
40  &  9.33e-05  &  4.36e-07  &  8.48e-12 \\
\multicolumn{2}{l}{Aspect ratio 3}\\
  \cline{1-3}
20  &  2.80e-01  &  3.93e-01  &  5.73e-01 \\
30  &  4.72e-02  &  2.64e-02  &  1.23e-03 \\
40  &  9.37e-05  &  1.76e-06  &  2.71e-11 \\
\multicolumn{2}{l}{Aspect ratio 4}\\
  \cline{1-3}
20  &  3.82e-01  &  2.47e-01  &  4.46e-01 \\
30  &  7.78e-02  &  2.88e-02  &  1.00e-03 \\
40  &  1.36e-04  &  3.52e-06  &  1.02e-10 \\
\multicolumn{2}{l}{Aspect ratio 5}\\
  \cline{1-3}
20  &  2.11e-01  &  1.95e-01  &  4.11e-01 \\
30  &  4.23e-02  &  3.07e-02  &  6.80e-04 \\
40  &  1.45e-04  &  1.91e-06  &  2.05e-09 \\
\end{rptable}

\begin{rptable}{Chebyshev particle}{$\lambda$}{table:hive1}
20  &  9.36e-03  &  9.26e-03  &  1.25e-02 \\
30  &  5.21e-04  &  6.04e-04  &  1.23e-03 \\
40  &  1.03e-04  &  9.92e-05  &  1.24e-04 \\
\end{rptable}

\begin{rptable}{Chebyshev particle}{$4 \lambda$}{table:hive4}
40  &  1.64e-03  &  8.77e-04  &  2.00e-04 \\
50  &  7.49e-05  &  5.15e-05  &  2.48e-05 \\
60  &  1.02e-05  &  6.00e-06  &  4.01e-06 \\
\end{rptable}

\begin{rptable}{Chebyshev particle}{$8 \lambda$}{table:hive8}
60  &  7.89e-03  &  1.19e-03  &  1.26e-02 \\
70  &  4.39e-04  &  8.74e-05  &  4.96e-04 \\
80  &  2.69e-05  &  5.16e-06  &  2.09e-05 \\
\end{rptable}

\begin{table}
  \revisedus
  \centering
  \parbox{8cm}{%
  \begin{tabular}{crr}
    \hline
    $n$ & memory & CPU time\\
    \hline
    20 & 107\,MB & 34.7\,s \\
    40 & 1.52\,GB & 6.8\,min \\
    60 & 7.43\,GB & 40.4\,min\\
    80 & 23.09\,GB & 3.0\,h\\ 
    \hline
  \end{tabular}}
  \caption{\label{table:cpu}
    Memory required to store the discretization matrix and CPU wall time
    for computing the far-field for 1202 distinct incident plane waves
    for the non-convex Chebyshev particles (with high-order accuracy indicated in Tables~\ref{table:hive1}--\ref{table:hive8}).
  CPU times are in parallel on a single compute node with a 2.8\,GHz Intel Xeon  8562Y 32-cores processor.}
\end{table}

\farfieldfig{Chebyshev particle}{\poly}{8}{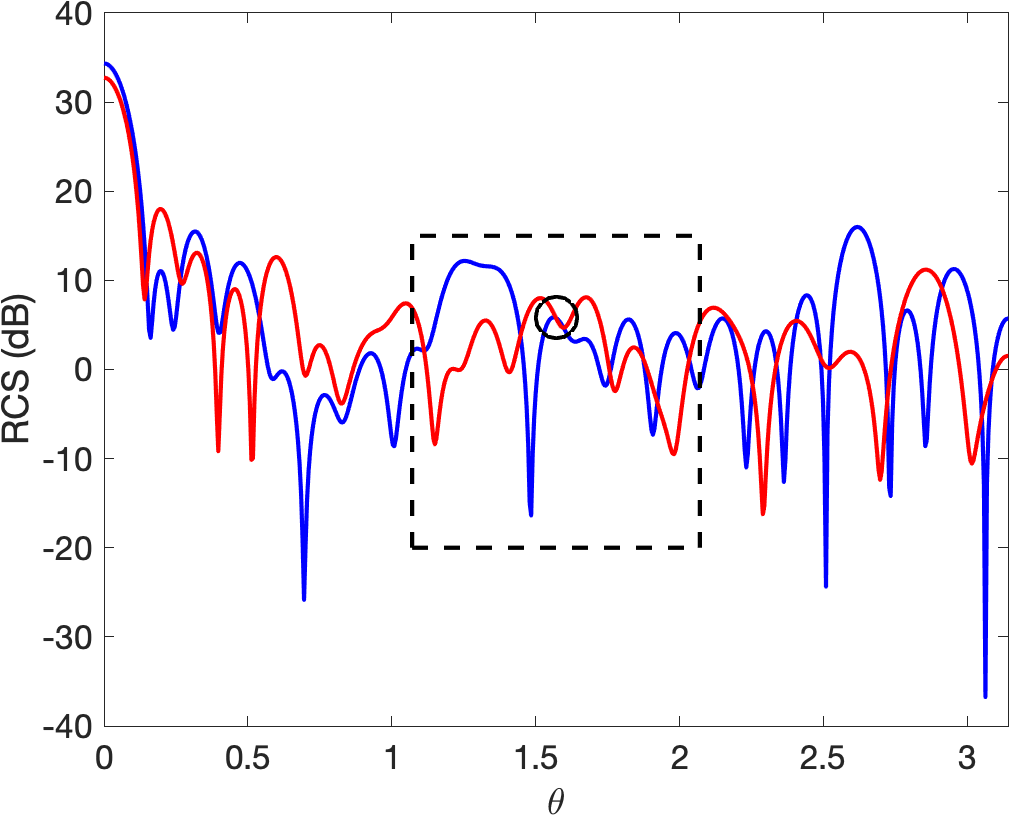}{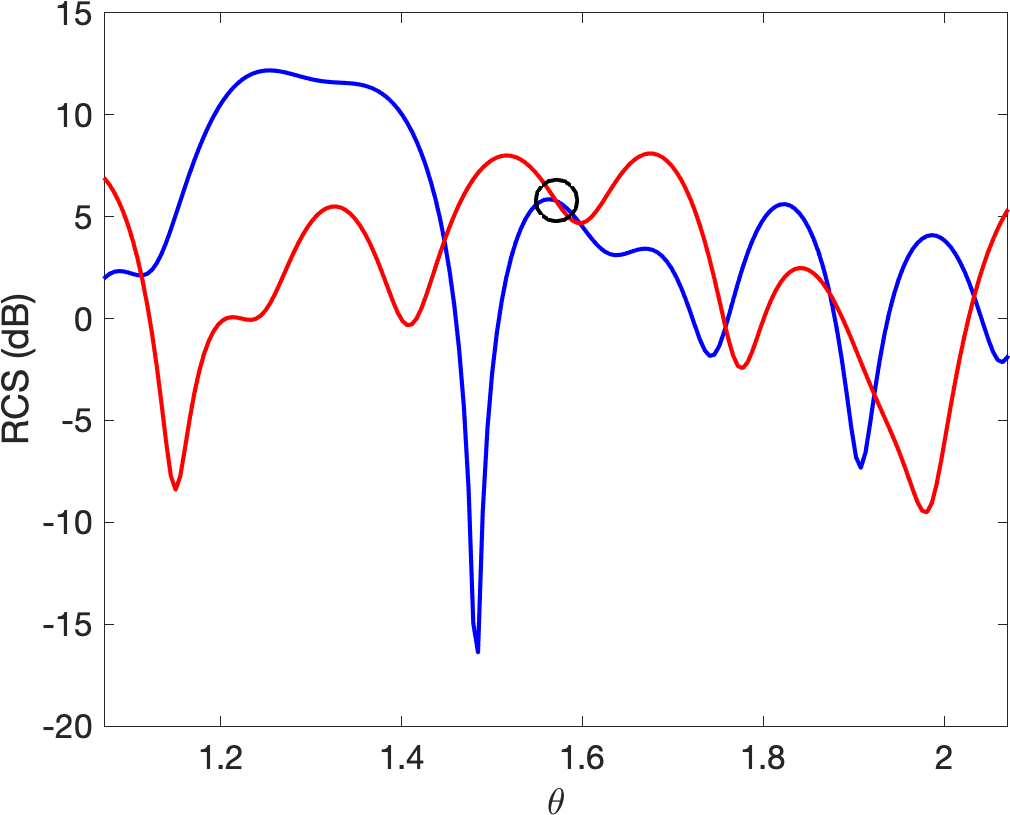}{fig:far-hive8-polycarbonate}{80}
\farfieldfig{Chebyshev particle}{\mishi}{8}{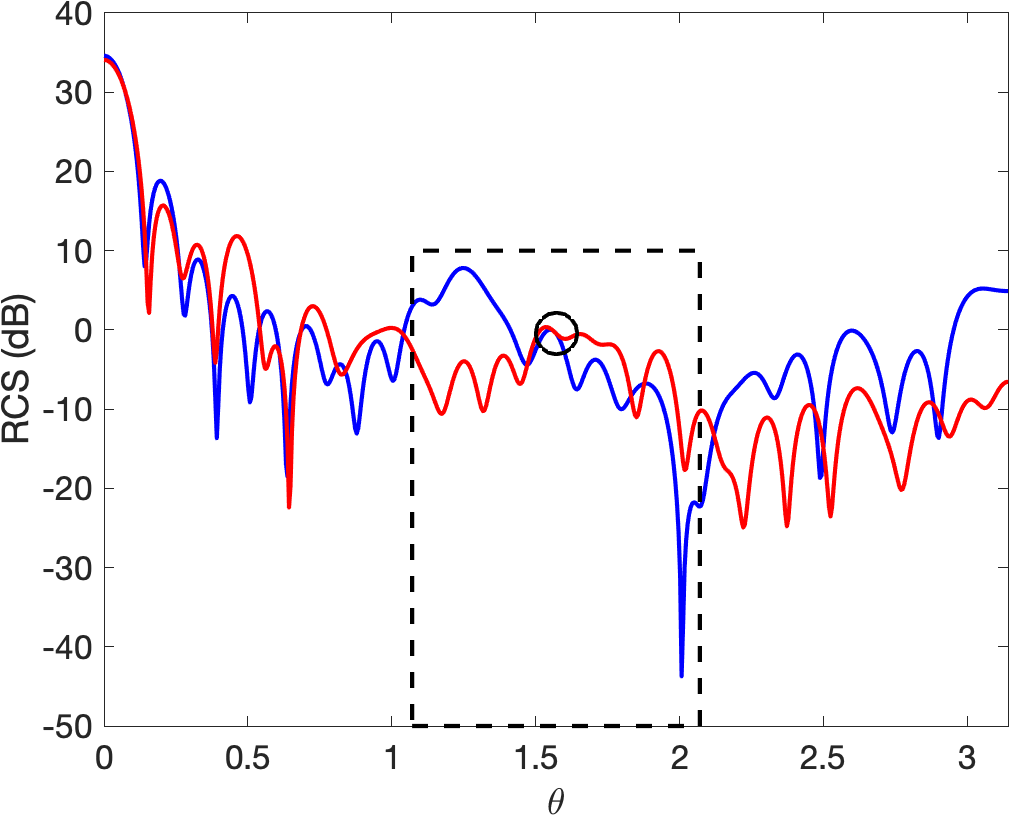}{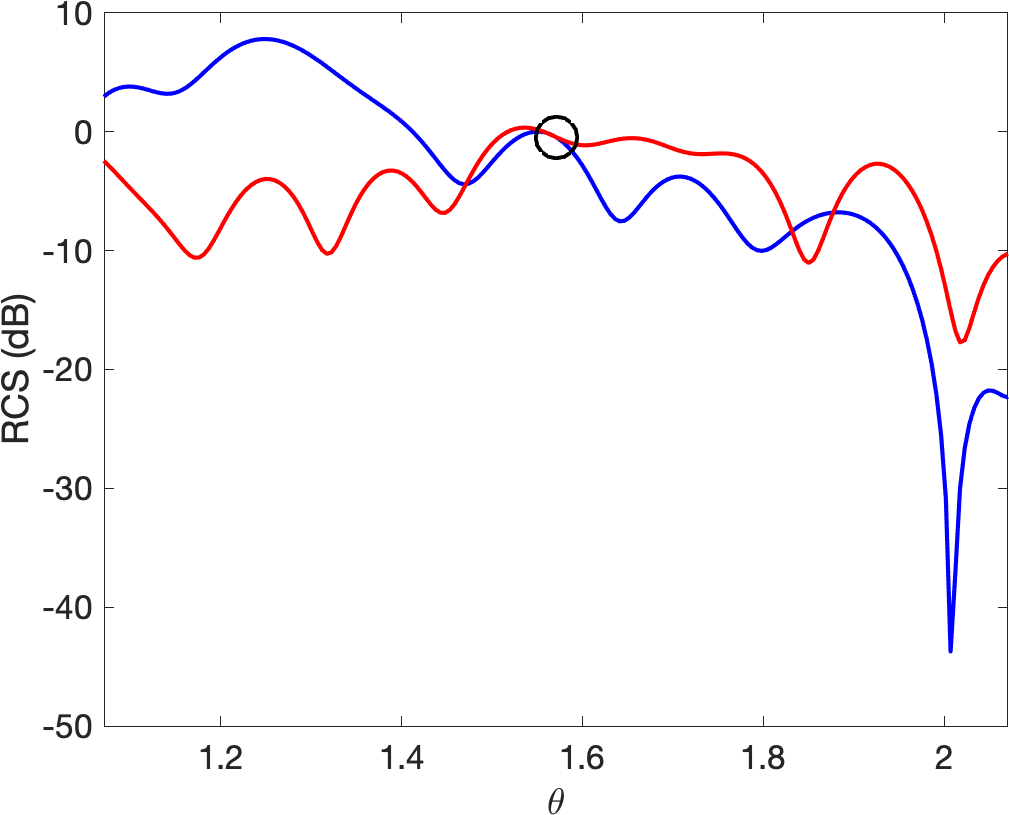}{fig:far-hive8-mishchenko1}{80}
\farfieldfig{Chebyshev particle}{\mishii}{8}{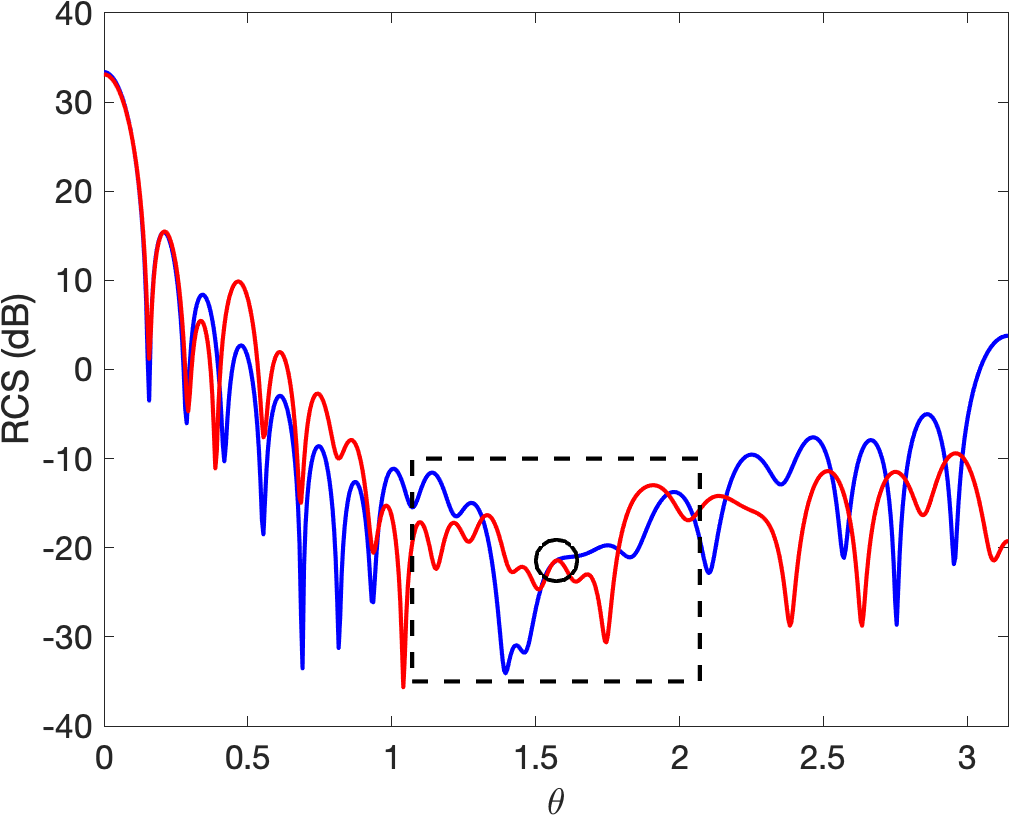}{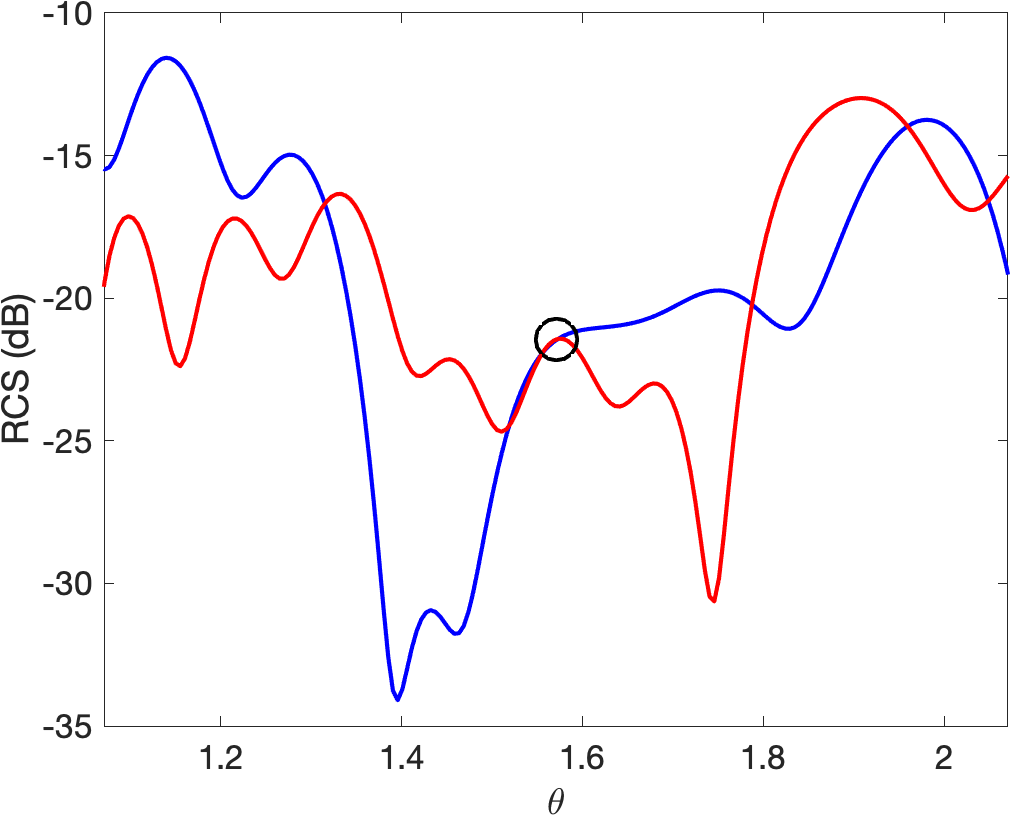}{fig:far-hive8-mishchenko2}{80}

\mishfig{2}{\poly}{\mishii}{8}{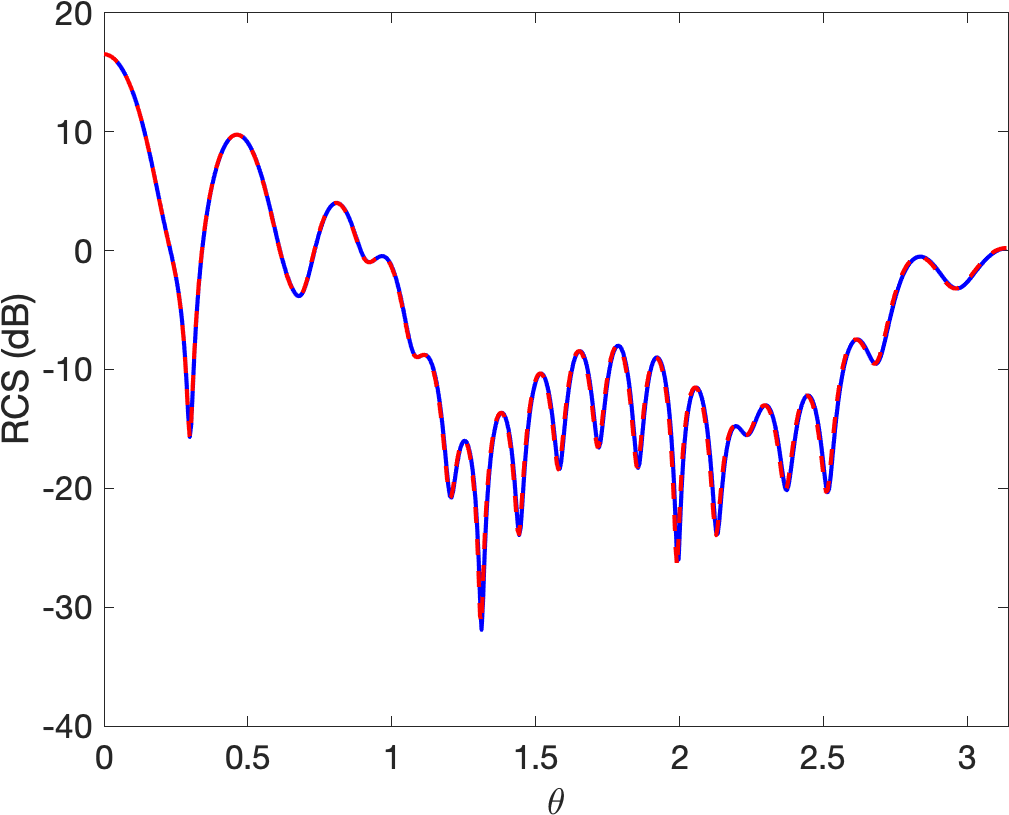}{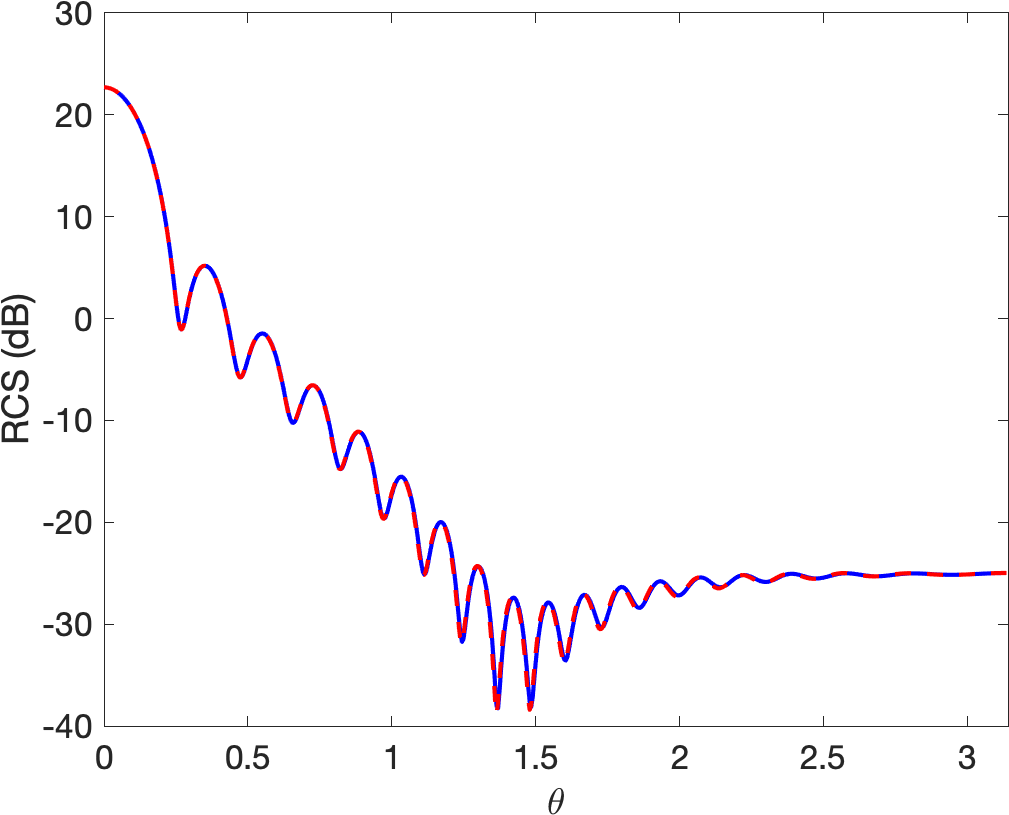}{fig:mish-ell2-size8}{50}

\mishfig{5}{\poly}{\mishii}{1}{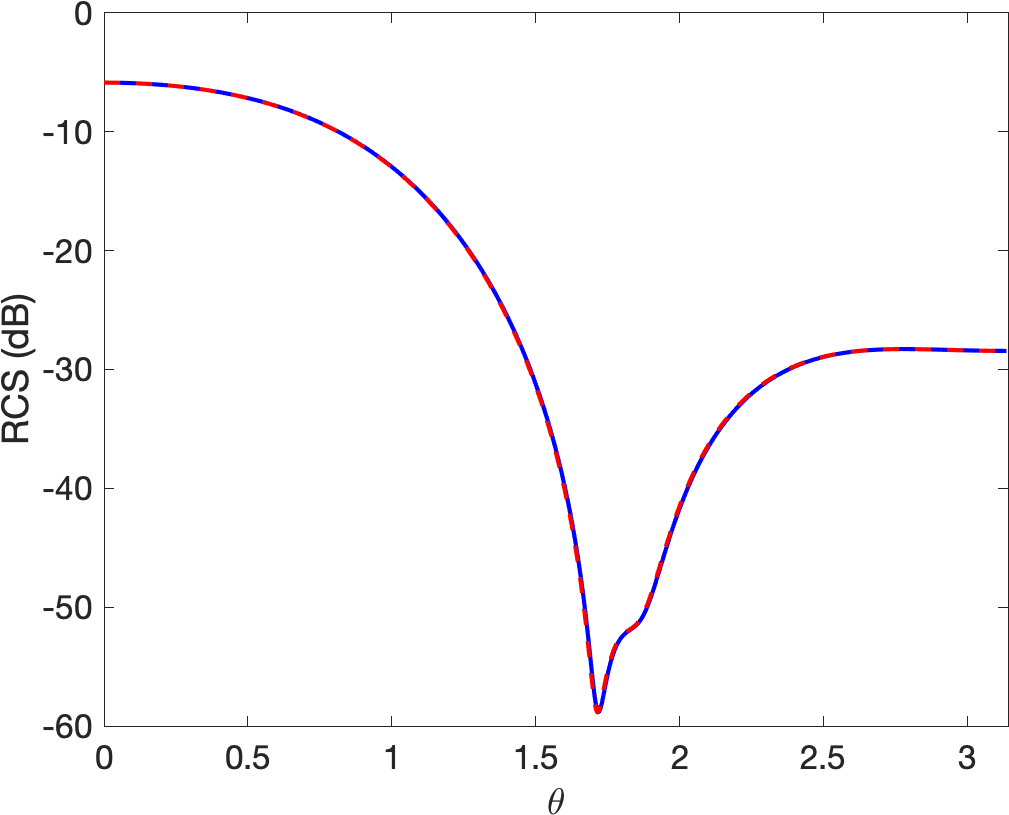}{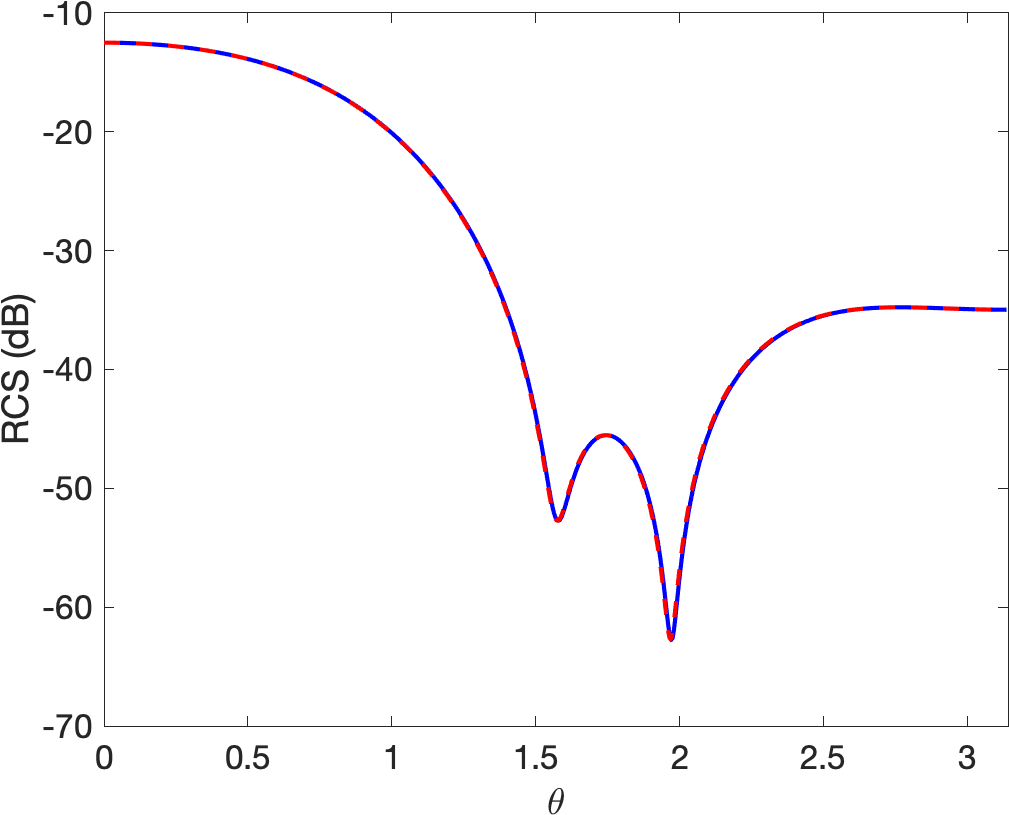}{fig:mish-ell5-size1}{30}

\renewcommand{\errorsuffix}{_error}
\extfieldfig{Chebyshev particle}{\mishii}{4}{8}{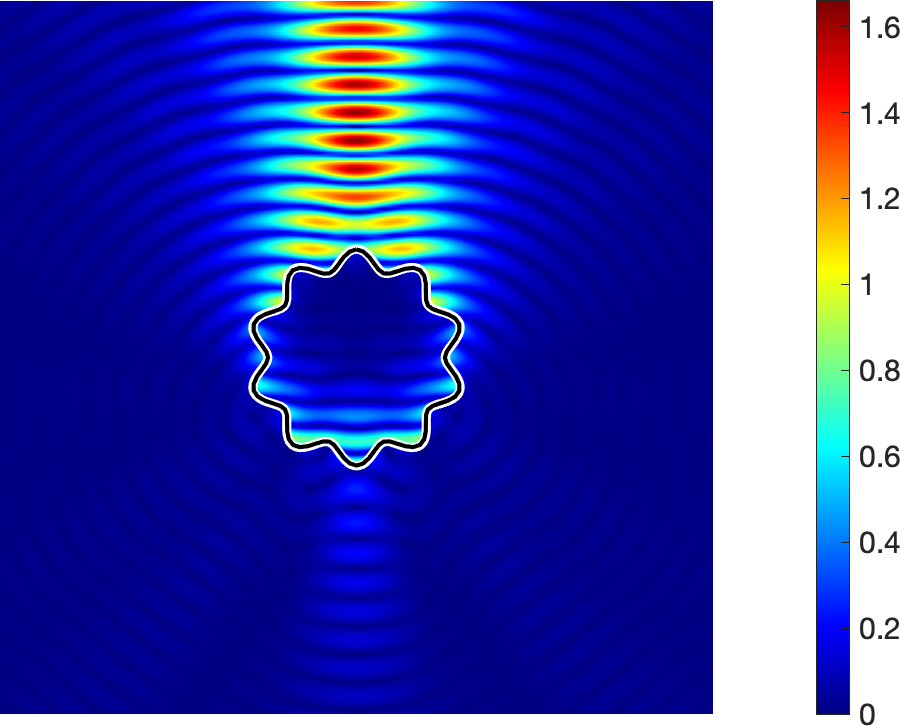}{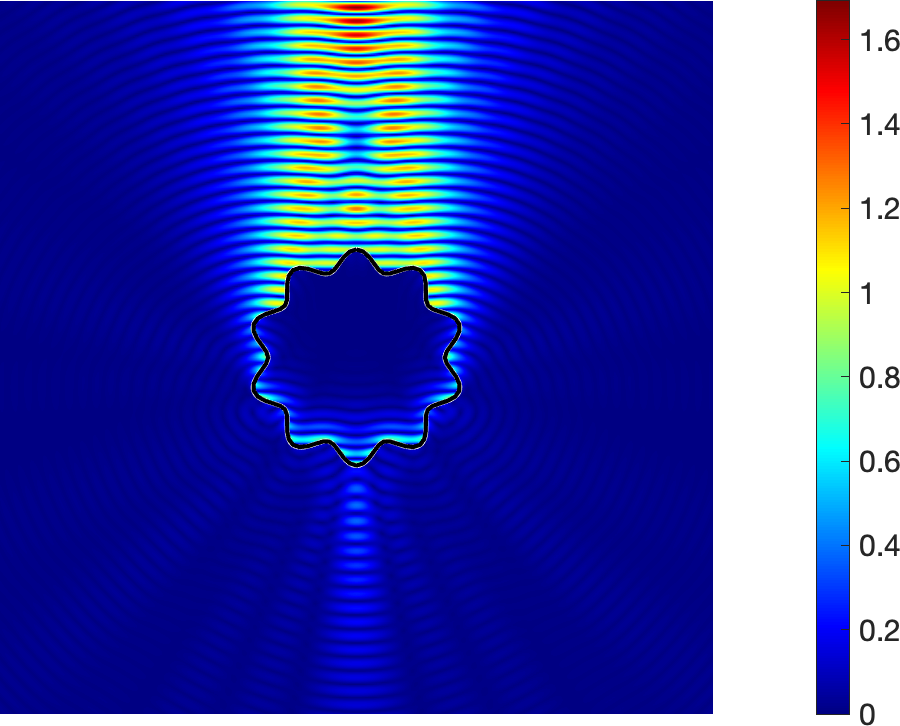}{fig:ext-hive8-mishchenko2}{80}{110}

  \begin{figure}
  \centering
    \revisedB
    \fbox{%
  \includegraphics[width=10cm]{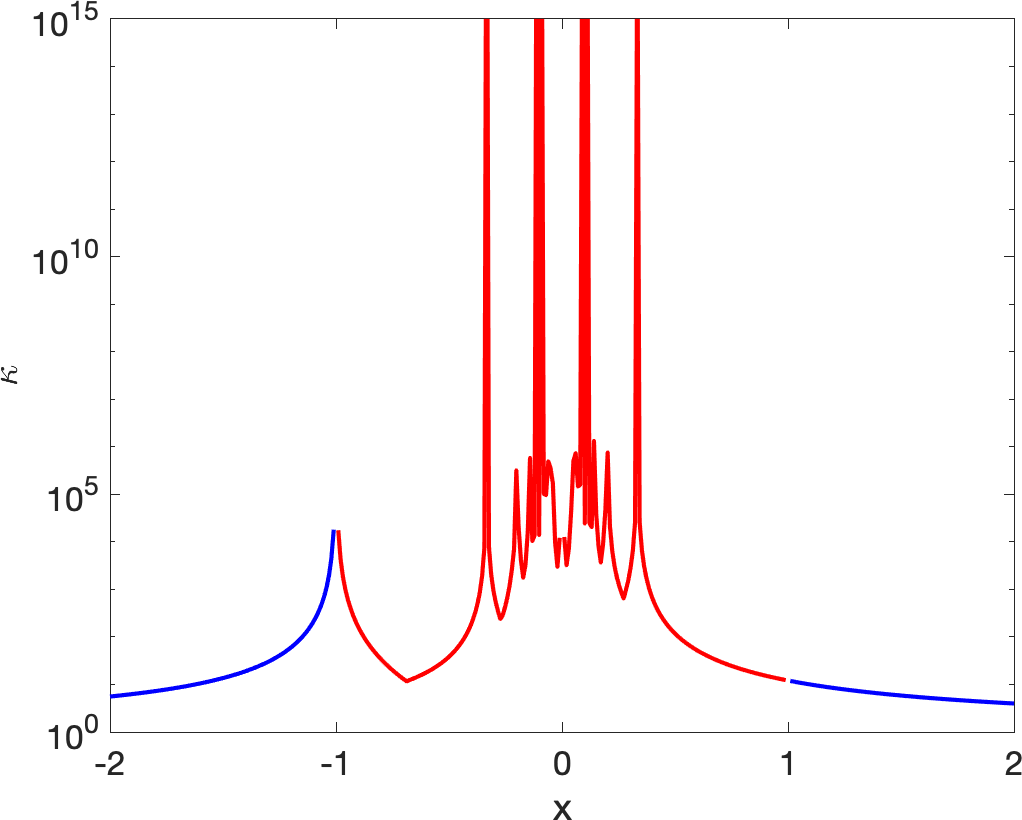}}
  \caption{\label{fig:cond-eddy}
    Condition number $\kappa$ of the matrix in system~\eqref{eq:SIE}
    (with stabilization)
    for a unit sphere with
    size $10^{-6}$
    against $x=(1+\widehat{\epsilon})/(1-\widehat{\epsilon})$
    where $\widehat{\epsilon} = \epsilon_-/\epsilon_+$ is the
    relative permittivity of the scatterer.
    The path along which the condition number is computed
    is shown in Figure~\ref{fig:path}.
    The stability analysis in this article applies to the 
    dielectric region $x \in (-\infty,-1) \cup (1,\infty)$
    (blue line).
    The plasmonic region is $x \in (-1,1)$ (red line).
    }
\end{figure}

  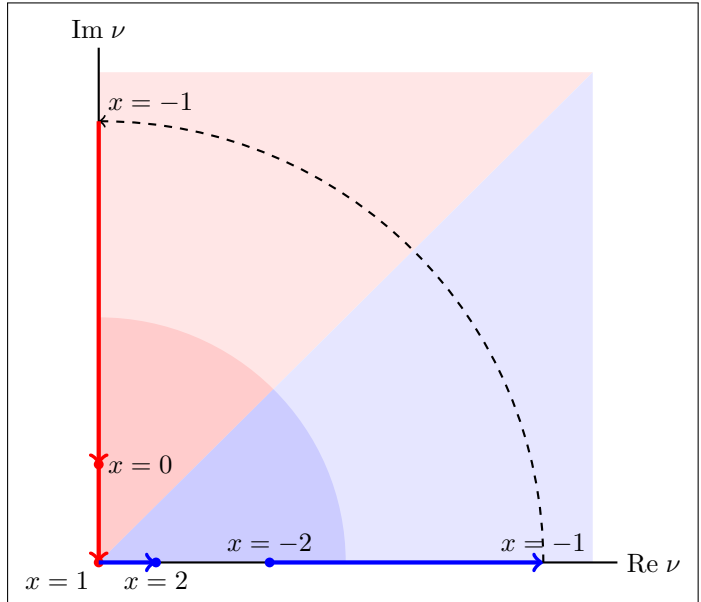
\begin{figure}
    \centering
    \revisedB
    \fbox{%
      \color{black}
\begin{tikzpicture}[scale=0.65]
  \fill[red!10!white] (0,0) -- (10,10) -- (0,10) -- cycle;
  \fill[blue!10!white] (0,0) -- (10,10) -- (10,0) -- cycle;
  \begin{scope}
  \clip (5,0) arc [start angle=0, end angle=90, radius=5] -- (0,0) -- cycle;
  \fill[red!20!white] (0,0) -- (10,10) -- (0,10) -- cycle;
  \fill[blue!20!white] (0,0) -- (10,10) -- (10,0) -- cycle;
  \end{scope}
  \draw[thick] (0,0) -- (10.5,0);
  \draw[thick] (0,0) -- (0,10.5);
  \draw (0,10.5) node[anchor=south] {$\mathrm{Im} \; \nu$};
  \draw (10.5,0) node[anchor=west] {$\mathrm{Re} \; \nu$};
  \draw[blue,ultra thick,->] (3.46,0) -- (9,0);
  \fill[blue] (3.46,0) circle[radius=0.1];
  \draw (3.46,0) node[anchor=south] {$x=-2$};
  \draw (9,0) node[anchor=south] {$x=-1$};
  \draw[thick,dashed,->] (9,0) arc[start angle=0,end angle=90,radius=9];
  \draw[red,ultra thick,->] (0,9) -- (0,2);
  \fill[red] (0,2) circle[radius=0.1];
  \draw (0,9) node[anchor=south west] {$x=-1$};
  \draw (0,2) node[anchor=west] {$x=0$};
  \draw[red,ultra thick,->] (0,2) -- (0,0);
  \fill[red] (0,0) circle[radius=0.1];
  \draw (0,0) node[anchor=north east] {$x=1$};
  \draw[blue,ultra thick,->] (0,0) -- (1.16,0);
  \fill[blue] (1.16,0) circle[radius=0.1];
  \draw (1.16,0) node[anchor=north] {$x=2$};
    \end{tikzpicture}}
    \caption{\label{fig:path}
      Schematic showing the
      path along which the condition number is plotted in
      Figure~\ref{fig:cond-eddy}.
      The 
      dielectric and absorbing region is marked by blue shading,
      the plasmonic region by red shading, and the Eddy current region
      where $|\nu|$ is large by light shading.
    }
\end{figure}

  \begin{eddytable}[\revisedB]{$10^{-5}\lambda$}{table:eddy0p00001}
2  &  4.09e-07  &  3.97e-05  &  8.62e-05 \\
3  &  1.50e-07  &  5.22e-06  &  5.79e-06 \\
4  &  3.26e-09  &  5.58e-07  &  2.32e-06 \\
\end{eddytable}

  \begin{eddytable}[\revisedB]{$10^{-4}\lambda$}{table:eddy0p0001}
2  &  3.95e-05  &  2.64e-03  &  2.65e+00 \\
3  &  1.50e-08  &  5.91e-08  &  3.76e-02 \\
4  &  4.48e-10  &  2.86e-08  &  2.41e-04 \\
\end{eddytable}

\begin{figure}
    \centering
    \revisedB
    \fbox{%
        \includegraphics[width=10cm]{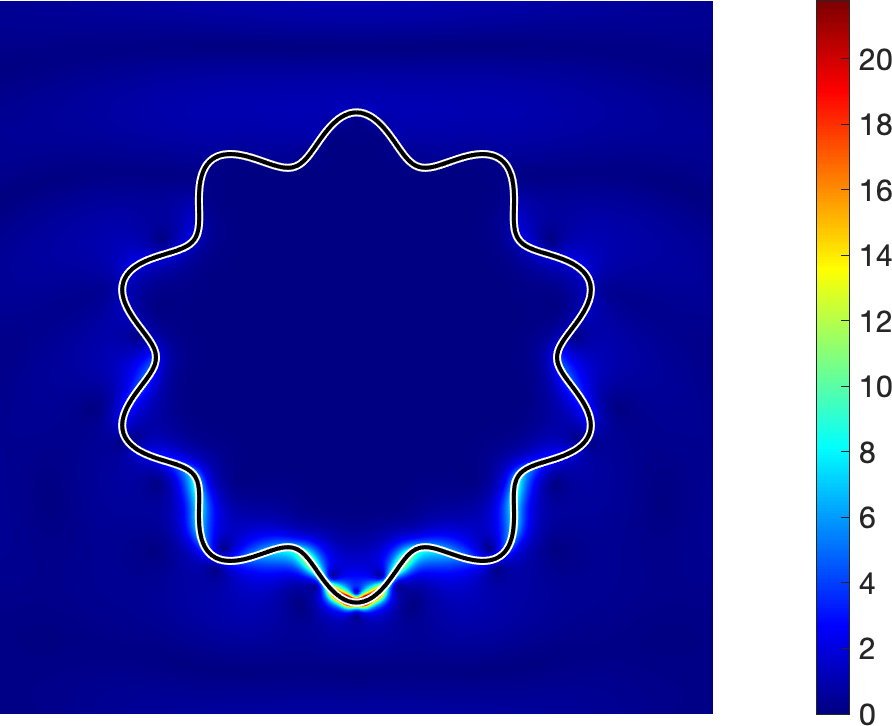}}
    \caption{\label{fig:ext-hive2-plasmonic}
      Visualization of the interior and exterior total electric fields
      $| \mathop{\mathrm{Re}} \E_{\tot,n}|$ 
      induced by a H-polarized plane wave
  striking a {\it plasmonic} Chebyshev particle with  refractive index $\nu = 1.088025734989757i$
  and electromagnetic size $s = 2$ 
  computed with $n=100$.
    }
\end{figure}

\clearpage
\appendix
\section{Singularity of standard coupling approach} \label{Appendix1}
\begin{prop}
Let $H$ be a Hilbert space. Assume that either
$H$ is  finite dimensional with $\dim H \geq 2$ or $H$ is separable.
 There exist two compact linear operators $\Mm, \Jm: H \to H$
 such that 
$N(\Jm) \cap N(\Ima +\Mm) = \{ 0 \}$ and, for every $\xi \in \mathbb{C}$,
$\Ima +\Mm + \xi \Jm $ is singular.
\end{prop}
\begin{proof}
In dimension 2 we set
\begin{equation}
\Ima+ \Mm= \left(
\begin{array}{cc}
1 & 0 \\
0 & 0
\end{array}
\right), \quad 
\Jm= \left(
\begin{array}{cc}
0 & 1 \\
0 & 0
\end{array}
\right) \nonumber
\end{equation}
It is clear that $\Ima + \Mm  + \xi \Jm $ is singular, for any $\xi \in \mathbb{C}$.
Generalizing this example to higher dimensions is trivial using block matrices.

Let $H$ be a separable Hilbert space and $\{e_i : i \in \mathbb{N} , i\geq 1\}$
a Hilbert basis of $H$.
Define $\Mm$ by 
\begin{displaymath}
    \Mm e_i= \left\{ \begin{array}{lll}
      -e_i, & \qquad & \mbox{for $i=2$},\\
      0, && \mbox{for $i \neq 2$,}
      \end{array} \right.
\end{displaymath}
and $\Jm$ by
\begin{displaymath}
    \Jm e_i= \left\{ \begin{array}{lll}
      e_1, & \qquad & \mbox{for $i=2$},\\
      0, && \mbox{for $i \neq 2$.}
      \end{array} \right.
\end{displaymath}%
A simple calculation can show that $N(\Jm) \cap N(\Ima +\Mm) = \{ 0 \}$
and $(\Ima +\Mm + \xi \Jm) (\xi e_1 - e_2) =0$
for any $\xi \in \bbC$. 
\end{proof}

\begin{remark}
In dimension three or higher, or in a separable Hilbert space,
$\Ima + \Mm  + \xi \Jm $ may be singular for all $\xi \in \mathbb{C} $,
even if both $\Mm$ and $\Jm$ are symmetric,
compact, 
 and $N(\Jm) \cap N(\Ima +\Mm) = \{ 0 \}$.
In dimension 3, set
\begin{equation}
\Ima+ \Mm= \left(
\begin{array}{ccc}
0 & 0 & 1 \\
0 & 0  & 1  \\
1 & 1 & 0
\end{array}
\right), \quad 
\Jm= \left(
\begin{array}{ccc}
1 & 0 & 0 \\
0 & -1 & 0 \\
0 & 0 & 0
\end{array}
\right). \nonumber
\end{equation}
We note that $N(\Jm) \cap N(\Ima +\Mm) = \{ 0 \}$
and yet $(\Ima + \Mm  + \xi \Jm) (1,-1,-\xi) =0$, for all $\xi \in \mathbb{C}$.
This counterexample can be easily extended to any separable Hilbert space.
\end{remark}

\bmhead{Acknowledgments}
 The first two authors gratefully acknowledge the support of the
Australian Research Council (ARC) Discovery Project Grant (DP220102243).
MG was also supported by the Simons Foundation  through grant 518882.
SCH thanks the Isaac Newton Institute for Mathematical Sciences
for hospitality during the program Mathematical Theory and Applications
of Multiple Wave Scattering and acknowledges support of the EPSRC
(Grant Number EP/R014604/1) and the Simons Foundation.

\section*{Declarations}
{\bf Conflict of interest:} The authors declare no competing interests.

\bibliography{references}

\end{document}